\documentclass[12pt,notitlepage,a4paper,reqno,ps]{amsart}
\usepackage{eurosym}
\usepackage{amssymb}
\usepackage{amstext}
\usepackage{color}
\usepackage{amsmath,amsthm,amssymb,color}

\newcommand{\medint}{-\kern  -,375cm\int}

\definecolor{ora}{rgb}{0.8,0.2,0.1}
\definecolor{vio}{rgb}{0.5,0,0.5}
\definecolor{gre}{rgb}{0.1,0.6,0}
\definecolor{verde}{rgb}{0,0.7,0.4}
\newenvironment{michelarev}{\color{azzurro}}{\color{black}}
\newcommand{\bmicr}{\begin{michelarev}}
\newcommand{\emicr}{\end{michelarev}}
\theoremstyle{plain}
\newtheorem{theorem}{Theorem}[section]

\newtheorem{lemma}[theorem]{Lemma}
\newtheorem{proposition}[theorem]{Proposition}
\theoremstyle{definition}

\theoremstyle{remark}

\theoremstyle{plain}

\numberwithin{equation}{section} \makeatletter

\renewcommand{\p@enumi}{\thesection.}
\makeatother  \allowdisplaybreaks
\keywords{Nonstandard growth conditions; $p,q$-growth; Degenerate ellipticity; Lipschitz continuity.}
\subjclass[2010]{49N60, 35J50}

\begin{document}
\title[Lipschitz regularity for non-uniformly elliptic integrals ]{Lipschitz
regularity for\\
degenerate elliptic integrals with $p,q$-growth}
\author[ G. Cupini -- P. Marcellini -- E. Mascolo -- A. Passarelli di Napoli]%
{ G. Cupini -- P. Marcellini -- E. Mascolo -- A. Passarelli di Napoli}
\address{Giovanni Cupini: Dipartimento di Matematica, Universit\`a di Bologna%
\\
Piazza di Porta S.Donato 5, 40126 - Bologna, Italy}
\email{giovanni.cupini@unibo.it}
\address{Paolo Marcellini and Elvira Mascolo: Dipartimento di Matematica
``U. Dini'', Universit\`a di Firenze\\
Viale Morgagni 67/A, 50134 - Firenze, Italy}
\email{paolo.marcellini@unifi.it}
\email{elvira.mascolo@unifi.it}
\address{Antonia Passarelli di Napoli: Dipartimento di Matematica e Appl.
``R. Caccioppoli''\\
Universit\`a di Napoli ``Federico II''\\
Via Cintia, 80126 Napoli, Italy}
\email{antonia.passarellidinapoli@unina.it}
\thanks{\textit{Acknowledgements.} The authors are members of GNAMPA (Gruppo
Nazionale per l'Analisi Matematica, la Probabilit\`a e le loro Applicazioni)
of INdAM (Istituto Nazionale di Alta Matematica)}

\begin{abstract}
We establish the local Lipschitz continuity and the higher differentiability
of vector-valued local minimizers of a class of energy integrals of the
Calculus of Variations. The main novelty is that we deal with possibly
degenerate energy densities with respect to the $x-$variable.
\end{abstract}

\maketitle



%

\section{Introduction}

The paper deals with the regularity of minimizers of integral functionals of
the Calculus of Variations of the form 
\begin{equation}
F(u)=\int_{\Omega }f(x,Du)\,dx  \label{funzionale}
\end{equation}%
where $\Omega \subset \mathbb{R}^{n}$, $n\geq 2$, is a bounded open set, $%
u:\Omega \rightarrow \mathbb{R}^{N}$, $N\geq 1$, is a Sobolev map. The main
feature of \eqref{funzionale} is the possible degeneracy of the lagrangian $%
f(x,\xi )$ with respect to the $x-$variable. We assume that the Carath\'{e}%
odory function $f=f\left( x,\xi \right) $ is convex and of class $C^{2}$
with respect to $\xi \in \mathbb{R}^{N\times n}$, 
with $f_{\xi \xi }\left( x,\xi \right) $, $f_{\xi x}\left( x,\xi \right) $
also Carath\'{e}odory functions and $f(\cdot ,0)\in L^{1}(\Omega )$. We
emphasize that the $N\times n$ matrix of the second derivatives $f_{\xi \xi
}\left( x,\xi \right) $ not necessarily is uniformly elliptic and it may
degenerate at some $x\in \Omega $.

In the vector-valued case $N>1$ minimizers of functionals with general
structure may lack regularity, see \cite{deg},\cite{sverak},\cite{moo-sa},
and it is natural to assume a modulus-gradient dependence for the energy
density; i.e. that there exists $g=g(x,t):\Omega \times \lbrack 0,+\infty
)\rightarrow \lbrack 0,+\infty )$ such that 
\begin{equation}
f(x,\xi )=g(x,|\xi |).  \label{(A1)}
\end{equation}%
Without loss of generality we can assume $g(x,0)=0$; indeed the minimizers
of $F$ are minimizers of $u\mapsto \int_{\Omega }\left(
f(x,Du)-f(x,0)\right) \,dx$ too. Moreover, by \eqref{(A1)} and the convexity
of $f$, $g\left( x,t\right) $ is a non-negative, convex and increasing
function of $t\in \left[ 0,+\infty \right) $.

As far as the growth and the ellipticity assumptions are concerned, we
assume that there exist exponents $p,q$, nonnegative measurable functions $%
a(x),k(x)$ and a constant $L>0$ such that 
\begin{equation}
\left\{ 
\begin{array}{l}
a\left( x\right) \,(1+|\xi |^{2})^{\frac{p-2}{2}}|\lambda |^{2}\leq \langle
f_{\xi \xi }(x,\xi )\lambda ,\lambda \rangle \leq L\,(1+|\xi |^{2})^{\frac{%
q-2}{2}}|\lambda |^{2},\quad 2\leq p\leq q, \\ 
\left\vert f_{\xi x}(x,\xi )\right\vert \leq k\left( x\right) (1+|\xi
|^{2})^{\frac{q-1}{2}}%
\end{array}%
\right.  \label{(A2)--(A4)}
\end{equation}%
for a.e. $x\in \Omega $ and for every $\xi ,\lambda \in \mathbb{R}^{N\times
n}$. We allow the coefficient $a\left( x\right) $ to be zero, so that (\ref{(A2)--(A4)})$_{1}$ is a not uniform ellipticity
condition. As proved in Lemma \ref{l:growth}, (\ref{(A2)--(A4)})$_{1}$
implies the following possibly degenerate $p,q-$growth conditions for $f$,
for some constant $c>0$, 
\begin{equation}
c\,a\left( x\right) (1+|\xi |^{2})^{\frac{p-2}{2}}|\xi |^{2}\leq f(x,\xi
)\leq L(1+|\xi |^{2})^{\frac{q}{2}},\;\;\;\;\text{a.e. }x\in \Omega
,\;\forall \;\xi \in \mathbb{R}^{N\times n}.  \label{(A2)--(A4)-I}
\end{equation}

%
%

Our main result concerns the local Lipschitz regularity and the higher
differentiability of the local minimizers of $F$.

\begin{theorem}
\label{t:main} Let the functional $F$ in \eqref{funzionale} satisfy %
\eqref{(A1)} and \eqref{(A2)--(A4)}. Assume moreover that 
\begin{equation}
\frac{1}{a}\in L_{\mathrm{loc}}^{s}(\Omega ),\qquad k\in L_{\mathrm{loc}%
}^{r}(\Omega ),  \label{kba}
\end{equation}%
with $r,s>n$ and 
\begin{equation}
\frac{q}{p}<\frac{s}{s+1}\left( 1+\frac{1}{n}-\frac{1}{r}\right) .
\label{gap}
\end{equation}%
If $u\in W_{\mathnormal{loc}}^{1,1}(\Omega )$ is a local minimizer of $F$, 
then for every ball $B_{R_{0}}\Subset \Omega $ the following estimates 
\begin{equation}
\Vert Du\Vert _{L^{\infty }(B_{R_{0}/2})}\leq C\mathcal{K}%
_{R_{0}}^{\vartheta }\left( \int_{B_{R_{0}}}\left( 1+f(x,Du)\right)
\,dx\right) ^{\vartheta }  \label{fin}
\end{equation}%
\begin{equation}
\int_{B_{R_{0}/2}}a(1+|Du|^{2})^{\frac{p-2}{2}}\left\vert D^{2}u\right\vert
^{2}\,dx\leq C\mathcal{K}_{R_{0}}^{\vartheta }\left( \int_{B_{R_{0}}}\left(
1+f(x,Du)\right) \,dx\right) ^{\vartheta },  \label{hdfin}
\end{equation}%
hold with the exponent $\vartheta $ depending on the data, the constant $C$
also depending on $R_{0}$ and where $\mathcal{K}_{R_{0}}=1+\Vert a^{-1}\Vert
_{L^{s}(B_{R_{0}})}\Vert k\Vert _{L^{r}(B_{R_{0}})}^{2}$.%
%
%
%
\end{theorem}

%

It is well known that to get regularity under $p,q-$growth the exponents $q$
and $p$ cannot be too far apart; usually, the gap between $p$ and $q$ is
described by a condition relating $p,q$ and the dimension $n$. In our case
we take into account the possible degeneracy of $a(x)$ and the condition (%
\ref{(A2)--(A4)})$_{2}$ on the mixed derivatives $f_{\xi x}$ in terms of a
possibly unbounded coefficient $k(x)$; then we deduce that the gap depends
on $s$, the summability exponent of $a^{-1}$ that \textquotedblleft
measures\textquotedblright\ how much $a$ is degenerate, and the exponent $r$
that tell us\ how far $k(x)$ is from being bounded. If $s=r=\infty $ then %
\eqref{gap} reduces to $\frac{q}{p}<1+\frac{1}{n}$ that is what one expects,
see \cite{cupguimas} and for instance \cite{MarcelliniJMAA2020}. Moreover,
if $s=\infty $ and $n<r\leq +\infty $, then \eqref{gap} reduces to $\frac{q}{%
p}<1+\frac{1}{n}-\frac{1}{r}$ and we recover the result of \cite{EMM1}.

Motivated by applications to the theory of elasticity, recently Colombo and
Mingione \cite{colmin},\cite{colmin2} (see also \cite{BCM},\cite{EMM3},\cite%
{defi-ming},\cite{defi-ok}) studied the so-called \textit{double phase
integrals} 
\begin{equation}
\int_{\Omega }|Du|^{p}+b(x)|Du|^{q}\,dx,\quad 1<p<q\,.  \label{funz_modello}
\end{equation}%
The model case we have in mind here is different: we consider the degenerate
functional with non standard growth of the form 
\begin{equation}
I(u)=\int_{B_{1}(0)}a(x)(1+|Du|^{2})^{\frac{p}{2}}+b(x)(1+|Du|^{2})^{\frac{q%
}{2}}\,dx  \label{functional-2}
\end{equation}%
with $0\leq a(x)\leq b(x)\leq L$ for some $L>0$. The integrand of $I(u)$
satisfies (\ref{(A2)--(A4)})$_{2}$ with $k\left( x\right) =\left\vert
Da\left( x\right) \right\vert +\left\vert Db\left( x\right) \right\vert $.
It is worth mentioning that in the literature $a(x)$ is usually assumed
positive and bounded away from zero, see e.g. \cite{BCM},\cite{colmin},
which is not the case here since $a(x)$ may vanish at some point. The
counterpart is that we consider the powers of $(1+|Du|^{2})^{\frac{1}{2}}$
instead of $|Du|$. We notice that the regularity result of Theorem \ref%
{t:main} is new also when $p=q\geq 2$, for example for the energy integral 
\begin{equation}
F_{1}(u)=\int_{B_{1}(0)}a(x)(1+|Du|^{2})^{\frac{p}{2}}\,dx
\label{integral with p}
\end{equation}%
with $a(x)\geq 0$, $\frac{1}{a}\in L^{s}(\Omega )$ and $\left\vert
Da\right\vert \in L^{r}$ with $\frac{1}{s}+\frac{1}{r}<\frac{1}{n}$. As far
as we know, the results proposed here are the first approach to the study of
the Lipschitz continuity of the local minimizers in the setting of
degenerate elliptic integrals under $p,q-$growth.

As well known, weak solutions to the elliptic equation in divergence form of
the type 
\begin{equation*}
-\mathrm{div}\left( A(x,Du)\right) =0\,\,\text{in}\,\,\Omega .
\end{equation*}%
are locally Lipschitz continuous provided the vector field $A:\Omega \times 
\mathbb{R}^{n}\rightarrow \mathbb{R}^{n}$ is differentiable with respect to $%
\xi $ and satisfies the uniformly elliptic conditions 
\begin{equation*}
\Lambda _{1}(1+|\xi |^{2})^{\frac{p-2}{2}}|\lambda |^{2}\leq \langle A_{\xi
}(x,\xi )\lambda ,\lambda \rangle \leq \Lambda _{2}(1+|\xi |^{2})^{\frac{p-2%
}{2}}|\lambda |^{2}.
\end{equation*}%
Trudinger \cite{Trud} started the study of the interior regularity of
solutions to linear elliptic equation of the form 
\begin{equation}
\sum_{i,j=1}^{n}\frac{\partial }{\partial x_{i}}\left( a_{ij}(x)\,\frac{%
\partial u}{\partial {x_{j}}}(x)\right) =0\,,\qquad x\in \Omega \subseteq 
\mathbb{R}^{n},  \label{linear-equation}
\end{equation}%
where the measurable coefficients $a_{ij}$ satisfy the non-uniform condition 
\begin{equation}
\lambda (x)|\xi |^{2}\leq \sum_{i,j=1}^{n}a_{ij}(x)\xi _{i}\xi _{j}\leq
n^{2}\mu (x)|\xi |^{2}  \label{ipotesi-trudinger}
\end{equation}%
for a.e. $x\in \Omega $ and every $\xi \in \mathbb{R}^{n}$. Here $\lambda
(x) $ is the minimum eigenvalue of the symmetric matrix $A(x)=(a_{ij}(x))$
and $\mu (x):=\sup_{ij}|a_{ij}|$. Trudinger proved that any weak solution of %
\eqref{linear-equation} is locally bounded in $\Omega $, under the following
integrability assumptions on $\lambda $ and $\mu $ 
\begin{equation}
\lambda ^{-1}\in L_{\mathrm{loc}}^{r}(\Omega )\quad \text{and}\quad \mu
_{1}=\lambda ^{-1}\mu ^{2}\in L_{\mathrm{loc}}^{\sigma }(\Omega )\quad \text{%
with $\frac{1}{r}+\frac{1}{\sigma }<\frac{2}{n}$}.
\label{risultato-trudinger}
\end{equation}%
The equation \eqref{linear-equation} is usually called \textit{degenerate}
when $\lambda ^{-1}\notin L^{\infty }(\Omega )$, whereas it is called 
\textit{singular} when $\mu \notin L^{\infty }(\Omega )$. These names in
this case refer to the \textit{degenerate} and the \textit{singular} cases
with respect to the $x-$variable, but in the mathematical literature these
names are often referred to the gradient variable; this happens for instance
with the $p-$\textit{Laplacian operator }$-\mathrm{div}\left( \left\vert
Du\right\vert ^{p-2}Du\right) $. We do not study in this paper the
degenerate case with respect to the gradient variable, but we refer for
instance to the analysis made by Duzaar and Mingione \cite%
{Duzaar-Mingione2010}, who studied an $L^{\infty }-$gradient bound for
solutions to non-homogeneous $p-$Laplacian type systems and equations; see
also Cianchi and Maz'ya \cite{Cianchi-Mazya2011} and the references therein
for the rich literature on the subject.

The result by Trudinger was extended in many settings and directions:
firstly, by Trudinger himself in \cite{Trud2} and later by Fabes, Kenig and
Serapioni in \cite{fabes-kenig-serapioni}; Pingen in \cite{pingen} dealt
with systems. More recently for the regularity of solutions and minimizers
we refer to \cite{bella-sch},\cite{BiaCupMas},\cite{cruz}\cite{CupMarMas17}, 
\cite{CupMarMas18},\cite{Iwaniec}. For the higher integrability of the
gradient we refer to \cite{Iwaniec-Sbordone} (see also \cite{CMP1}). Very
recently Calderon-Zygmund's estimates for the $p-$Laplace operator with
degenerate weights have been established in \cite{bdgp}. The literature
concerning non-uniformly elliptic problems is extensive and we refer the
interested reader to the references therein.

The study of the Lipschitz regularity in the $p,q-$growth context started
with the papers by Marcellini \cite{mar89},\cite{mar91} and, since then,
many and various contributions to the subject have been provided, see the
references in \cite{Mingione},\cite{MarcelliniJMAA2020}. The vectorial
homogeneous framework was considered in \cite{mar96},\cite{mar-papi} and by
Esposito, Leonetti and Mingione \cite{espleomin},\cite{espleominp}. The
condition (\ref{(A2)--(A4)})$_{2}$ for general non autonomous integrands $%
f=f(x,Du)$ has been first introduced in \cite{EMM1},\cite{EMM2},\cite{EMM3}.
It is worth to highlight that, due to the $x-$dependence, the study of
regularity is significantly harder and the techniques more complex. The
research on this subject is intense, as confirmed by the many articles
recently published, see e.g. \cite{CGGP},\cite{Cupini-Marcellini-Mascolo2012}%
,\cite{DiMarco-Marcellini},\cite{GPdN}, \cite%
{MarcelliniDiscrContDinSystems2019},\cite{MarcelliniNonAnal2019},\cite%
{MarcelliniJMAA2020},\cite{PdN14-1},\cite{PdN14-2},\cite{PdN15}.

Let us briefly sketch the tools to get our regularity result. First, for
Lipschitz and higher differentiable minimizers, we prove a weighted
summability result for the second order derivatives of minimizers of
functionals with possibly degenerate energy densities, see Proposition \ref%
{weighted}. Next in Theorem \ref{t:apriori} we get an \textit{a-priori
estimate} for the $L^{\infty }$-norm of the gradient. To establish the 
\textit{a-priori estimate} we use the Moser's iteration method \cite{moser}
for the gradient and the ideas of Trudinger \cite{Trud}. An approximation
procedure allows us to conclude. Actually, if $u$ is a local minimizer of %
\eqref{funzionale}, we construct a sequence of suitable variational problems
in a ball $B_{R}\subset \subset \Omega $ with boundary value data $u$. In
order to apply the a-priori estimate to the minimizers of the approximating
functionals we prove a higher differentiability result (Theorem \ref%
{highdiff}) for minimizers of the class of functionals with $p,q-$growth
studied in \cite{EMM1}, where only the Lipschitz continuity was proved. By
applying the previous a-priori estimate to the sequence of the solutions we
obtain a uniform control in $L^{\infty }$ of the gradient which allows to
transfer the local Lipschitz continuity property to the original minimizer $%
u $.

Another difficulty due to the $x-$dependence of the energy density is that
the \textit{Lavrentiev phenomenon} may occur. A local minimizer of $F$ is a
function $u\in W_{\mathrm{loc}}^{1,1}(\Omega )$ such that $f(x,Du)\in L_{%
\mathrm{loc}}^{1}(\Omega )$ and 
\begin{equation*}
\int_{\Omega }f(x,Du)\,dx\leq \int_{\Omega }f(x,Du+D\varphi )\,dx
\end{equation*}%
for every $\varphi \in C_{0}^{1}(\Omega )$. If $u$ is a local minimizer of
the functional $F$, by virtue of \eqref{(A2)--(A4)-I} we have that $%
a(x)|Du|^{p}\in L_{\mathrm{loc}}^{1}(\Omega )$ and, by \eqref{kba}, $u\in W_{%
\mathrm{loc}}^{1,\frac{ps}{s+1}}(\Omega )$ since 
\begin{equation}
\int_{B_{R}}|Du|^{\frac{ps}{s+1}}\,dx\leq \left(
\int_{B_{R}}a|Du|^{p}\,dx\right) ^{\frac{s}{s+1}}\left( \int_{B_{R}}\frac{1}{%
a^{s}}\,dx\right) ^{\frac{1}{s+1}}<+\infty   \label{intu}
\end{equation}%
for every ball $B_{R}\subset \Omega $. Therefore in our context a-priori the
presence of the Lavrentiev phenomenon cannot be excluded. Indeed, due to the
growth assumptions on the energy density, the integral in \eqref{funzionale}
is well defined if $u\in W^{1,q\frac{r}{r-1}}$, but a-priori this is not the
case if $u\in W^{1,\frac{ps}{s+1}}(\Omega )\setminus W_{\mathrm{loc}}^{1,q%
\frac{r}{r-1}}(\Omega )$. However, as a consequence of Theorem \ref{t:main},
under the stated assumptions \eqref{(A1)},\eqref{(A2)--(A4)},(\ref{kba}),(%
\ref{gap}) the Lavrentiev phenomenon for the integral functional $F$ in %
\eqref{funzionale} cannot occur. For the gap in the Lavrentiev phenomenon we
refer to \cite{zhikov},\cite{belloni-buttazzo},\cite{espleominp},\cite%
{esp-leo-pet}.

We conclude this introduction by observing that even in the one-dimensional
case the Lipschitz continuity of minimizers for non-uniformly elliptic
integrals is not obvious. Indeed, if we consider a minimizer $u$ to the
one-dimensional integral 
\begin{equation}
F\left( u\right) =\int_{-1}^{1}a\left( x\right) \left\vert u^{\prime }\left(
x\right) \right\vert ^{p}\,dx\,,\qquad p>1,  \label{one-dimensional integral}
\end{equation}%
then the Euler's first variation takes the form 
\begin{equation*}
\int_{-1}^{1}a\left( x\right) p\left\vert u^{\prime }\left( x\right)
\right\vert ^{p-2}u^{\prime }\left( x\right) \varphi ^{\prime }\left(
x\right) \,dx\,=0,\;\;\;\forall \;\varphi \in C_{0}^{1}\left( -1,1\right) .
\end{equation*}%
This implies that the quantity $a\left( x\right) \left\vert u^{\prime
}\left( x\right) \right\vert ^{p-2}u^{\prime }\left( x\right) $ is constant
in $\left( -1,1\right) $; it is a nonzero constant, unless $u\left( x\right) 
$ itself is constant in $\left( -1,1\right) $, a trivial case that we do not
consider here. In particular the sign of $u^{\prime }\left( x\right) $ is
constant and we get 
\begin{equation*}
\left\vert u^{\prime }\left( x\right) \right\vert ^{p-1}=\frac{c}{a\left(
x\right) }\,,\;\;\;\text{a.e.}\;x\in \left( -1,1\right) .
\end{equation*}%
Therefore if $a\left( x\right) $ vanishes somewhere in $\left( -1,1\right) $
then $\left\vert u^{\prime }\left( x\right) \right\vert $ is unbounded (and
viceversa), independently of the exponent $p>1$. Thus for $n=1$ the local
Lipschitz regularity of the minimizers does not hold in general if the
coefficient $a\left( x\right) $ vanishes somewhere.

We can compare this one-dimensional fact with the general conditions
considered in the Theorem \ref{t:main}. In the case $a\left( x\right)
=\left\vert x\right\vert ^{\alpha }$ for some $\alpha \in \left( 0,1\right) $
then, taking into account the assumptions in (\ref{kba}), for the integral
in (\ref{one-dimensional integral}) we have $k\left( x\right) =a^{\prime
}\left( x\right) =\alpha \left\vert x\right\vert ^{\alpha -2}x$ and 
\begin{equation*}
\left\{ 
\begin{array}{ccccc}
\frac{1}{a}\in L_{\mathrm{loc}}^{s}\left( -1,1\right) & \Leftrightarrow & 
1-\alpha s>0 & \Leftrightarrow & \alpha <\frac{1}{s} \\ 
k\left( x\right) =a^{\prime }\in L_{\mathrm{loc}}^{r}\left( -1,1\right) & 
\Leftrightarrow & r\left( \alpha -1\right) >-1 & \Leftrightarrow & \alpha >1-%
\frac{1}{r}.%
\end{array}%
\right.
\end{equation*}%
These conditions are compatible if and only if $1-\frac{1}{r}<\frac{1}{s}$.
Therefore, also in the one-dimensional case we have a counterexample to the $%
L^{\infty }-$gradient bound in (\ref{fin}) if 
\begin{equation}
\frac{1}{r}+\frac{1}{s}>1\,.  \label{condition for the counterexample}
\end{equation}%
This is a condition that can be easily compared with the assumption (\ref%
{gap}) for the validity of $L^{\infty }-$gradient bound (\ref{fin}) in the
general $n-$dimensional case. In fact, being $1\leq \frac{q}{p}$, (\ref{gap}%
) implies 
\begin{equation*}
1<\frac{s}{s+1}\left( 1+\frac{1}{n}-\frac{1}{r}\right)
\;\;\;\;\Leftrightarrow \;\;\;\;\frac{1}{r}+\frac{1}{s}<\frac{1}{n}\,,
\end{equation*}%
which essentially is the complementary condition to (\ref{condition for the
counterexample}) when $n=1$.

The plan of the paper is the following. In Section \ref{s:preliminary} we
list some definitions and preliminary results. In Section \ref{s:apriori} we
prove an a-priori estimates of the $L^{\infty }$-norm of the gradient of
local minimizers and an higher differentiability result, see Theorem \ref%
{t:apriori}. In Section \ref{s:highdiff} we prove an estimate for the second
order derivatives of a minimizer of an auxiliary uniformly elliptic
functional, see Theorem \ref{highdiff}. In the last section we complete the
proof of Theorem \ref{t:main}.

\section{Preliminary results}

\label{s:preliminary}

We shall denote by $C$ or $c$ a general positive constant that may vary on
different occasions, even within the same line of estimates. Relevant
dependencies will be suitably emphasized using parentheses or subscripts. In
what follows, $B(x,r)=B_{r}(x)=\{y\in \mathbb{R}^{n}:\,\,|y-x|<r\}$ will
denote the ball centered at $x$ of radius $r$. We shall omit the dependence
on the center and on the radius when no confusion arises.



To prove our higher differentiability result ( see Theorem \ref{highdiff}
below) we use the finite difference operator. For a function $u:\Omega
\rightarrow \mathbb{R}^{k}$, $\Omega $ open subset of $\mathbb{R}^{n}$,
given $s\in \{1,\ldots ,n\}$, we define 
\begin{equation}
\tau _{s,h}u(x):=u(x+he_{s})-u(x),\qquad x\in \Omega _{|h|},
\label{e:rappincr}
\end{equation}%
where $e_{s}$ is the unit vector in the $x_{s}$ direction, $h\in \mathbb{R}$
and 
\begin{equation*}
\Omega _{|h|}:=\{x\in \Omega \,:\,\mathrm{dist\,}(x,\partial \Omega )<|h|\}.
\end{equation*}%
We now list the main properties of this operator.

\begin{itemize}
\item[(i)] if $u\in W^{1,t}(\Omega )$, $1\leq t\leq \infty $, then $\tau
_{s,h}u\in W^{1,t}(\Omega _{|h|})$ \and%
\begin{equation*}
D_{i}(\tau _{s,h}u)=\tau _{s,h}(D_{i}u),
\end{equation*}

\item[(ii)] if $f$ or $g$ has support in $\Omega _{|h|}$, then 
\begin{equation*}
\int_{\Omega }f\tau _{s,h}g\,dx=\int_{\Omega }g\tau _{s,-h}f\,dx,
\end{equation*}

\item[(iii)] if $u,u_{x_{s}}\in L^{t}(B_{R})$, $1\leq t<\infty $, and $%
0<\rho <R$, then for every $h$, $|h|\leq R-\rho $, 
\begin{equation*}
\int_{B_{\rho }}|\tau _{s,h}u(x)|^{t}\,dx\leq
|h|^{t}\int_{B_{R}}|u_{x_{s}}(x)|^{t}\,dx,
\end{equation*}

\item[(iv)] if $u\in L^{t}(B_{R})$, $1<t<\infty $, and for $0<\rho <R$ there
exists $K>0$ such that for every $h$, $|h|<R-\rho $, 
\begin{equation}
\sum_{s=1}^{n}\int_{B_{\rho }}|\tau _{s,h}u(x)|^{t}\,dx\leq K|h|^{t},
\label{e:rho}
\end{equation}%
then letting $h$ go to $0$, $Du\in L^{t}(B_{\rho })$ and $\Vert
u_{x_{s}}\Vert _{L^{t}(B_{\rho })}\leq K$ for every $s\in \{1,\ldots ,n\}$. 
\end{itemize}

We recall the following estimate for the auxiliary function 
\begin{equation}
V_{p}(\xi ):=\Bigl(1+|\xi |^{2}\Bigr)^{\frac{p-2}{4}}\xi ,  \label{Vfunction}
\end{equation}%
which is a convex function since $p\geq 2$ (see the Step 2 in \cite{mar85}
and the proof of \cite[Lemma 8.3]{Giusti}).

%
%

\begin{lemma}
\label{Vi} Let $1<p<\infty$. There exists a constant $c=c(n,p)>0$ such that 
\begin{equation*}
c^{-1}\Bigl(1 +| \xi |^2+| \eta |^2 \Bigr)^{\frac{p-2}{2}}\leq \frac{%
|V_{p}(\xi )-V_{p}(\eta )|^2}{|\xi -\eta |^2} \leq c\Bigl( 1 +|\xi |^2+|\eta
|^2 \Bigr)^{\frac{p-2}{2}}
\end{equation*}
for any $\xi$, $\eta \in \mathbb{R}^n$.
\end{lemma}

In the next lemma we prove that (\ref{(A2)--(A4)})$_{1}$ implies the,
possibly degenerate, $p,q-$growth condition stated in \eqref{(A2)--(A4)-I}.

\begin{lemma}
\label{l:growth} Let $f=f(x,\xi )$ be convex and of class $C^{2}$ with
respect to the $\xi -$variable.

Assume \eqref{(A1)} and 
\begin{equation}
a(x)\,(1+|\xi |^{2})^{\frac{p-2}{2}}|\lambda |^{2}\leq \langle D_{\xi \xi
}f(x,\xi )\lambda ,\lambda \rangle \leq b(x)\,(1+|\xi |^{2})^{\frac{q-2}{2}%
}|\lambda |^{2}  \label{(A2)}
\end{equation}%
for some exponents $2\leq p\leq q$ and nonnegative functions $a,b$. Then
there exists a constant $c$ such that 
\begin{equation*}
c\,a(x)(1+|\xi |^{2})^{\frac{p-2}{2}}|\xi |^{2}\leq f(x,\xi )\leq
b(x)(1+|\xi |^{2})^{\frac{q}{2}}+f(x,0).
\end{equation*}
\end{lemma}

\begin{proof}
For $x\in \Omega $ and $s\in \mathbb{R}$, let us set $\varphi (s)=g(x,st)$,
where we recall that $g$ is linked to $f$ by \eqref{(A1)}. The assumptions
on $f$ imply that $\varphi \in C^{2}(\mathbb{R})$ and that $g_{t}$ is
increasing in the gradient variable $t\in \lbrack 0;+\infty )$ with $%
g_{t}(x,0)=0$. Since 
\begin{equation*}
\varphi ^{\prime }(s)=g_{t}(x,st)\cdot t\,,\;\;\;\;\;\varphi ^{\prime \prime
}(s)=g_{tt}(x,st)\cdot t^{2},
\end{equation*}%
Taylor expansion formula yields that there exists $\vartheta \in (0,1)$ such
that 
\begin{equation*}
\varphi (1)=\varphi (0)+\varphi ^{\prime }(0)+\frac{1}{2}\varphi ^{\prime
\prime }(\vartheta )\,.
\end{equation*}%
Recalling the definition of $\varphi $, we get 
\begin{equation}  \label{g}
g(x,t)=g(x,0)+g_{t}(x,0)\cdot t+\frac{1}{2}g_{tt}(x,\vartheta t)\cdot
t^{2}=g(x,0)+\frac{1}{2}g_{tt}(x,\vartheta t)\cdot t^{2}.
\end{equation}%
Assumption \eqref{(A2)} translates into 
\begin{equation}
a(x)\,(1+t^{2})^{\frac{p-2}{2}}\leq g_{tt}(x,t)\leq b(x)\,(1+t^{2})^{\frac{%
q-2}{2}}.  \label{(A1b)}
\end{equation}%
Inserting \eqref{(A1b)} in \eqref{g}, we obtain 
\begin{equation}
\frac{a(x)}{2}\,(1+(\vartheta t)^{2})^{\frac{p-2}{2}}t^{2}+g(x,0)\leq
g(x,t)\leq g(x,0)+\frac{b(x)}{2}\,[1+(\vartheta t)^{2}]^{\frac{q-2}{2}}t^{2}.
\label{(A1c)}
\end{equation}%
Note that, since $\vartheta <1$ and $q>2$, the right hand side of %
\eqref{(A1c)} can be controlled with 
\begin{equation*}
g(x,t)\leq g(x,0)+b(x)\,[1+(\vartheta t)^{2}]^{\frac{q-2}{2}}t^{2}\leq
g(x,0)+b(x)\,[1+t^{2}]^{\frac{q-2}{2}}t^{2}
\end{equation*}%
Moreover since $g(x,0)\geq 0$ and $p\geq 2$, the left hand side of %
\eqref{(A1c)} can be controlled from below as follows 
\begin{equation*}
g(x,t)\geq \frac{a(x)}{2}\,(1+(\vartheta t)^{2})^{\frac{p-2}{2}%
}t^{2}+g(x,0)\geq \frac{a(x)}{2}\,(1+(\vartheta t)^{2})^{\frac{p-2}{2}}t^{2}
\end{equation*}%
\begin{equation*}
\geq \frac{a(x)}{2}\,(\vartheta ^{2}+(\vartheta t)^{2})^{\frac{p-2}{2}%
}t^{2}=\vartheta ^{p-2}\frac{a(x)}{2}(1+t^{2})^{\frac{p-2}{2}}t^{2}.
\end{equation*}%
Combining the last two estimates and recalling that $f(x,\xi )=g(x,|\xi |)$,
we conclude that there exists a constant $c=c(\vartheta )$ such that 
\begin{equation*}
c(\vartheta )a(x)(1+|\xi |^{2})^{\frac{p-2}{2}}|\xi |^{2}\leq f(x,\xi )\leq
b(x)\,(1+|\xi |^{2})^{\frac{q-2}{2}}|\xi |^{2}+f(x,0)
\end{equation*}%
and the conclusion follows.
\end{proof}


We end this preliminary section with a well known property. The following
lemma has important applications in the so called \textit{hole-filling method%
}. Its proof can be found for example in \cite[Lemma 6.1]{Giusti} . \medskip

\begin{lemma}
\label{iter} Let $h:[r,R_{0}]\rightarrow \mathbb{R}$ be a nonnegative
bounded function and $0<\vartheta <1$, $A,B\geq 0$ and $\beta >0$. Assume
that 
\begin{equation*}
h(s)\leq \vartheta h(t)+\frac{A}{(t-s)^{\beta }}+B,
\end{equation*}%
for all $r\leq s<t\leq R_{0}$. Then 
\begin{equation*}
h(r)\leq \frac{cA}{(R_{0}-r)^{\beta }}+cB,
\end{equation*}%
where $c=c(\vartheta ,\beta )>0$.
\end{lemma}

\bigskip

\section{The a-priori estimate\label{s:apriori}}

\medskip

The main result in this section is an a-priori estimate of the $L^{\infty }$%
-norm of the gradient of local minimizers of the functional $F$ in %
\eqref{funzionale} satisfying weaker assumptions than those in Theorem \ref%
{t:apriori}. 
Precisely, in this section we consider the following growth conditions 
\begin{equation}
\left\{ 
\begin{array}{l}
a(x)\,(1+|\xi |^{2})^{\frac{p-2}{2}}|\lambda |^{2}\leq \langle f_{\xi \xi
}(x,\xi )\lambda ,\lambda \rangle \leq b(x)\,(1+|\xi |^{2})^{\frac{q-2}{2}%
}|\lambda |^{2} \\ 
|f_{\xi x}(x,\xi )|\leq k(x)(1+|\xi |^{2})^{\frac{q-1}{2}},%
\end{array}%
\right.   \label{growth2(A2)--(A4)bis}
\end{equation}%
for a.e. $x\in \Omega $ and for every $\xi ,\lambda \in \mathbb{R}^{N\times
n}$. Here, $a,b,k$ are non-negative measurable functions. 
We do not require $a,b\in L^{\infty }$, but, in the main result of this
section, see Theorem \ref{t:apriori}, we assume the following summability
properties: 
\begin{equation}
\frac{1}{a}\in L_{\mathrm{loc}}^{s}(\Omega ),\qquad a\in L_{\mathrm{loc}}^{%
\frac{rs}{2s+r}}(\Omega ),\qquad b,\,k\in L_{\mathrm{loc}}^{r}(\Omega
),\qquad \text{with $r>n$}.  \label{kba2}
\end{equation}%
Moreover, we assume \eqref{gap}. 
We use the following weighted Sobolev type inequality, whose proof relies on
the H\"{o}lder's inequality, see e.g. \cite{CupMarMas18}.

\begin{lemma}
\label{Sob} Let $p\ge 2$, $s\ge 1$ and $w\in W_0^{1,\frac{ps}{s+1}}(\Omega;%
\mathbb{R}^N)$ ($w\in W_0^{1,p}(\Omega;\mathbb{R}^N)$ if $s=\infty$). Let $%
\lambda:\Omega\to [0,+\infty)$ be a measurable function such that $%
\lambda^{-1}\in L^{s}(\Omega)$. There exists a constant $c=c(n)$ such that 
\begin{equation}
\left(\int_\Omega|w|^{\sigma^*}\,dx\right)^{\frac{p}{\sigma^*}}\le
c(n)\|\lambda^{-1}\|_{L^s(\Omega)}\int_\Omega \lambda |Dw|^p\,dx,
\end{equation}
where $\sigma=\frac{ps}{s+1}$ ($\sigma=p$ if $s=+\infty$).
\end{lemma}

In establishing the a-priori estimate, we need to deal with quantities that
involve the $L^{2}$-norm of the second derivatives of the minimizer weighted
with the function $a(x)$. Next result tells that a $W^{2,\frac{2s}{s+1}}$
assumption on the second derivatives implies that they belong to the
weighted space $L^{2}(a(x)dx)$. More precisely, we have 

\begin{proposition}
\label{weighted} Consider the functional $F$ in \eqref{funzionale}
satisfying the assumption \eqref{growth2(A2)--(A4)bis} with 
\begin{equation}
a,b\in L_{\mathrm{loc}}^{1}(\Omega ),\,k\in L_{\mathnormal{loc}}^{\frac{2s}{%
s-1}}(\Omega ),  \label{kba2prop}
\end{equation}%
for some $s\geq 1$. If $u\in W_{\mathnormal{loc}}^{1,\infty }(\Omega )\cap
W_{\mathnormal{loc}}^{2,\frac{2s}{s+1}}(\Omega )$ is a local minimizer of $F$
then 
\begin{equation*}
a(x)|D^{2}u|^{2}\in L_{\mathrm{loc}}^{1}(\Omega ).
\end{equation*}%
%
\end{proposition}

\begin{proof}
Since $u$ is a local minimizer of the functional $F$, then $u$ satisfies the
Euler's system 
\begin{equation*}
\int_{\Omega}\sum_{i,\alpha}f_{\xi_i^\alpha}(x,Du)\varphi_{x_i}^{\alpha}(x)%
\,dx=0\qquad \forall \varphi\in C_{0}^{\infty}(\Omega;\mathbb{R}^N),
\end{equation*}
and, using the second variation, for every $s=1,\ldots,n$ it holds 
\begin{equation}
\int_{\Omega}\left\{
\sum_{i,j,\alpha,\beta}f_{\xi_{i}^{\alpha}\xi_{j}^{\beta}}(x,Du)%
\varphi_{x_i}^{\alpha}u_{x_sx_j}^{\beta}+\sum_{i,\alpha}
f_{\xi_i^{\alpha}x_s}(x,Du)\varphi_{x_i}^{\alpha}\right \}\,dx=0\qquad
\forall \varphi\in C_{0}^{\infty}(\Omega;\mathbb{R}^N).  \label{3.3o}
\end{equation}
Fix $s=1,\ldots,n$, a cut off function $\eta\in C_0^{\infty}(\Omega)$ and
define for any $\gamma\ge 0$ the function 
\begin{equation*}
\varphi^{\alpha}:=\eta^4u_{x_s}^{\alpha}\quad \alpha=1,\ldots,N.
\end{equation*}
Thanks to our assumptions on the minimizer $u$, through a standard density
argument, we can use $\varphi$ as test function in the equation \eqref{3.3o}%
, thus getting 
\begin{align*}
0=&\int_{\Omega}4\eta^3
\sum_{i,j,s,\alpha,\beta}f_{\xi_{i}^{\alpha}\xi_{j}^{\beta}}(x,Du)\eta_{x_i}
u_{x_s}^{\alpha}u_{x_sx_j}^{\beta}\,dx  \notag \\
& +\int_{\Omega}\eta^4
\sum_{i,j,s,\alpha,\beta}f_{\xi_{i}^{\alpha}\xi_{j}^{\beta}}(x,Du)
u_{x_sx_i}^{\alpha}u_{x_sx_j}^{\beta}\,dx  \notag \\
& +\int_{\Omega}4\eta^3\sum_{i,s,\alpha}
f_{\xi_{i}^{\alpha}x_s}(x,Du)\eta_{x_i}u_{x_s}^{\alpha}\,dx  \notag \\
& +\int_{\Omega}\eta^4 \sum_{i,s,\alpha}f_{\xi_{i}^{\alpha}x_s}(x,Du)
u_{x_sx_i}^{\alpha}\,dx  \notag \\
=:& J_{1}+J_2+J_3+J_4.
\end{align*}
By the use of Cauchy-Schwartz and Young's inequalities and by virtue of the
second inequality of \eqref{growth2(A2)--(A4)bis}, we can estimate the
integral $I_1$ as follows 
\begin{eqnarray*}
|J_1| &\le & 4\int_{\Omega}\!\! \Bigg\{\eta^2\sum_{i,j,s,\alpha,\beta}f_{%
\xi_{i}^{\alpha}\xi_{j}^{\beta}}(x,Du)\eta_{x_i} u_{x_s}^{\alpha} \eta_{x_j}
u_{x_s}^{\beta}\Bigg\}^{\frac{1}{2}}\!\! \Bigg\{\eta^4\sum_{i,j,s,\alpha,%
\beta}f_{\xi_{i}^{\alpha}\xi_{j}^{\beta}}(x,Du)
u_{x_sx_i}^{\alpha}u_{x_sx_j}^{\beta}\Bigg\}^{\frac{1}{2}} \\
&\le & C \int_{\Omega}\eta^2|D\eta|^2 b(x)(1+|Du|^2)^{\frac{q}{2}}\,dx \\
&&\qquad+ \frac{1}{2} \int_{\Omega}\eta^4
\sum_{i,j,s,\alpha,\beta}f_{\xi_{i}^{\alpha}\xi_{j}^{\beta}}(x,Du)
u_{x_sx_i}^{\alpha}u_{x_sx_j}^{\beta} \,dx.
\end{eqnarray*}

\noindent Moreover, by the last inequality in \eqref{growth2(A2)--(A4)bis}
we obtain 
\begin{eqnarray*}
|J_3|&\le& 4\int_{\Omega}\eta^3 k(x)(1+|Du|^2)^{\frac{q-1}{2}%
}\sum_{i,s,\alpha}|\eta_{x_i}u_{x_s}^{\alpha}|\,dx \\
&\le& 4\int_{\Omega}\eta^3|D\eta| k(x)(1+|Du|^2)^{\frac{q}{2}}\,dx.
\end{eqnarray*}
and also 
\begin{eqnarray*}
|J_4| &\le& \int_{\Omega}\eta^4 k(x)(1+|Du|^2)^{\frac{q-1}{2}}|D^2u|\,dx.
\end{eqnarray*}
Therefore we get 
\begin{eqnarray*}
&&\int_{\Omega} \eta^4
\sum_{i,j,s,\alpha,\beta}f_{\xi_{i}^{\alpha}\xi_{j}^{\beta}}(x,Du)
u_{x_sx_i}^{\alpha}u_{x_sx_j}^{\beta}\,dx\le \frac{1}{2}\int_{\Omega} \eta^4
\sum_{i,j,s,\alpha,\beta}f_{\xi_{i}^{\alpha}\xi_{j}^{\beta}}(x,Du)
u_{x_sx_i}^{\alpha}u_{x_sx_j}^{\beta}\,dx \\
&&+C \int_{\Omega}\eta^2|D\eta|^2 b(x)(1+|Du|^2)^{\frac{q}{2}%
}\,dx+4\int_{\Omega}\eta^3|D\eta| k(x)(1+|Du|^2)^{\frac{q}{2}}\,dx \\
&&+\int_{\Omega}\eta^4 k(x)(1+|Du|^2)^{\frac{q-1}{2}}|D^2u|\,dx.
\end{eqnarray*}
Reabsorbing the first integral in the right hand side by the left hand side
we obtain 
\begin{eqnarray*}
&&\int_{\Omega} \eta^4
\sum_{i,j,s,\alpha,\beta}f_{\xi_{i}^{\alpha}\xi_{j}^{\beta}}(x,Du)
u_{x_sx_i}^{\alpha}u_{x_sx_j}^{\beta}\,dx \\
&\le& C \int_{\Omega}\eta^2|D\eta|^2 b(x)(1+|Du|^2)^{\frac{q}{2}%
}\,dx+4\int_{\Omega}\eta^3|D\eta| k(x)(1+|Du|^2)^{\frac{q}{2}}\,dx \\
&&+\int_{\Omega}\eta^4 k(x)(1+|Du|^2)^{\frac{q-1}{2}}|D^2u|\,dx.
\end{eqnarray*}
By the ellipticity assumption in \eqref{growth2(A2)--(A4)bis} and since $%
u\in W^{1,\infty}_{\mathrm{loc}}(\Omega)$ we get 
\begin{eqnarray*}
&&\int_{\Omega} \eta^4 a(x)(1+|Du|^2)^{\frac{p-2}{2}}|D^2u|^2\,dx \\
&\le& C\|1+|Du|\|_{L^\infty(\mathrm{supp}\eta)}^{q}
\int_{\Omega}\eta^2|D\eta|^2 b(x)\,dx+C\|1+|Du|\|_{L^\infty(\mathrm{supp}%
\eta)}^{q}\int_{\Omega}\eta^3|D\eta|k(x) \\
&&+C\|1+|Du|\|_{L^\infty(\mathrm{supp}\eta)}^{q-1}\int_{\Omega}\eta^4
k(x)|D^2u|\,dx.
\end{eqnarray*}
H\"older's inequality yields 
\begin{eqnarray}
&&\int_{\Omega} \eta^4 a(x)(1+|Du|^2)^{\frac{p-2}{2}}|D^2u|^2\,dx  \notag \\
&\le& C\|1+|Du|\|_{L^\infty(\mathrm{supp}\eta)}^{q}
\int_{\Omega}\eta^2|D\eta|^2 b(x)\,dx+C\|1+|Du|\|_{L^\infty(\mathrm{supp}%
\eta)}^{q}\int_{\Omega}\eta^3|D\eta|k(x)  \notag \\
&&+C\|1+|Du|\|_{L^\infty(\mathrm{supp}\eta)}^{q-1}\left(\int_{\Omega}\eta^4
k^{\frac{2s}{s-1}}\,dx\right)^{\frac{s-1}{2s}}\left(\int_{\Omega}%
\eta^4|D^2u|^{\frac{2s}{s+1}}\,dx\right)^{\frac{s+1}{2s}}.  \label{stima2d}
\end{eqnarray}
Since $k\in L^{\frac{2s}{s-1}}_{\mathrm{loc}}(\Omega)$ and $u\in W^{2,\frac{%
2s}{s+1}}_{\mathnormal{loc}}(\Omega)$ then estimate \eqref{stima2d} implies
that 
\begin{equation*}
a(x)|D^2u|^2\in L^1_{\mathrm{loc}}(\Omega).
\end{equation*}
\end{proof}

\medskip We are now ready to establish the main result of this section.


\begin{theorem}
\label{t:apriori} Consider the functional $F$ in \eqref{funzionale}
satisfying the assumptions \eqref{growth2(A2)--(A4)bis}, \eqref{kba2}, %
\eqref{(A1)} and \eqref{gap}. 
If $u\in W_{\mathnormal{loc}}^{1,\infty }(\Omega )\cap W_{\mathnormal{loc}%
}^{2,{\frac{2s}{s+1}}}(\Omega )$ is a local minimizer of $F$ 
then for every ball $B_{R_{0}}\Subset \Omega $ 
\begin{equation}
\Vert Du\Vert _{L^{\infty }(B_{R_{0}/2})}\leq C\mathcal{K}%
_{R_{0}}^{\vartheta }\left( \int_{B_{R_{0}}}\left( 1+f(x,Du)\right)
\,dx\right) ^{\vartheta }  \label{stimafin}
\end{equation}%
\begin{equation}
\int_{B_{\rho }}a(1+|Du|^{2})^{\frac{p-2}{2}}\left\vert D^{2}u\right\vert
^{2}\,dx\leq c\left( {\int_{B_{R_{0}}}}(1+f(x,Du)\,dx\right) ^{\vartheta },
\label{stimafin2}
\end{equation}%
hold, for any $\rho <\frac{R}{2}$. Here 
\begin{equation*}
\mathcal{K}_{R_{0}}=1+\Vert a^{-1}\Vert _{L^{s}(B_{R_{0}})}\Vert k+b\Vert
_{L^{r}(B_{R_{0}})}^{2}+\Vert a\Vert _{L^{\frac{rs}{2s+r}}(B_{R_{0}})},
\end{equation*}%
$\vartheta >0$ is depending on the data, $C$ is depending also on $R_{0}$
and $c$ is depending also on $\rho $ and $\mathcal{K}_{R_{0}}$.
\end{theorem}

\begin{proof}
Since $u$ is a local minimizer of the functional $F$, then $u$ satisfies the
Euler's system 
\begin{equation*}
\int_{\Omega}\sum_{i,\alpha}f_{\xi_i^\alpha}(x,Du)\varphi_{x_i}^{\alpha}(x)%
\,dx=0\qquad \forall \varphi\in C_{0}^{\infty}(\Omega;\mathbb{R}^N),
\end{equation*}
and, using the second variation, for every $s=1,\ldots,n$ it holds 
\begin{equation}
\int_{\Omega}\left\{
\sum_{i,j,\alpha,\beta}f_{\xi_{i}^{\alpha}\xi_{j}^{\beta}}(x,Du)%
\varphi_{x_i}^{\alpha}u_{x_sx_j}^{\beta}+\sum_{i,\alpha}
f_{\xi_i^{\alpha}x_s}(x,Du)\varphi_{x_i}^{\alpha}\right \}\,dx=0\qquad
\forall \varphi\in C_{0}^{\infty}(\Omega;\mathbb{R}^N).  \label{3.3}
\end{equation}
Fix $s=1,\ldots,n$, a cut off function $\eta\in C_0^{\infty}(\Omega)$ and
define for any $\gamma\ge 0$ the function 
\begin{equation*}
\varphi^{\alpha}:=\eta^4u_{x_s}^{\alpha}(1+|Du|^2)^{\frac{\gamma}{2}}\quad
\alpha=1,\ldots,N.
\end{equation*}
One can easily check that 
\begin{align*}
\varphi^{\alpha}_{x_i}=& 4\eta^3\eta_{x_i}u_{x_s}^{\alpha}(1+|Du|^2)^{\frac{%
\gamma}{2}}+ \eta^ 4u_{x_sx_i}^{\alpha}(1+|Du|^2)^{\frac{\gamma}{2}}
+\gamma\eta^4u_{x_s}^{\alpha}(1+|Du|^2)^{\frac{\gamma-2}{2}}|Du|\big(|Du|%
\big)_{x_i}.
\end{align*}
Thanks to our assumptions on the minimizer $u$, through a standard density
argument, we can use $\varphi$ as test function in the equation \eqref{3.3},
thus getting 
\begin{align}
0=&\int_{\Omega}4\eta^3 (1+|Du|^2)^{\frac{\gamma}{2}}\sum_{i,j,s,\alpha,%
\beta}f_{\xi_{i}^{\alpha}\xi_{j}^{\beta}}(x,Du)\eta_{x_i}
u_{x_s}^{\alpha}u_{x_sx_j}^{\beta}\,dx  \notag \\
& +\int_{\Omega}\eta^4 (1+|Du|^2)^{\frac{\gamma}{2}}\sum_{i,j,s,\alpha,%
\beta}f_{\xi_{i}^{\alpha}\xi_{j}^{\beta}}(x,Du)
u_{x_sx_i}^{\alpha}u_{x_sx_j}^{\beta}\,dx  \notag \\
& +\gamma\int_{\Omega}\eta^4 (1+|Du|^2)^{\frac{\gamma-2}{2}%
}|Du|\sum_{i,j,s,\alpha,\beta}f_{\xi_{i}^{\alpha}\xi_{j}^{\beta}}(x,Du)
u_{x_s}^{\alpha}u_{x_sx_j}^{\beta} (|Du|)_{x_i} \,dx  \notag \\
& +\int_{\Omega}4\eta^3(1+|Du|^2)^{\frac{\gamma}{2}}\sum_{i,s,\alpha}
f_{\xi_{i}^{\alpha}x_s}(x,Du)\eta_{x_i}u_{x_s}^{\alpha}\,dx  \notag \\
& +\int_{\Omega}\eta^4 (1+|Du|^2)^{\frac{\gamma}{2}}\sum_{i,s,\alpha}f_{%
\xi_{i}^{\alpha}x_s}(x,Du) u_{x_sx_i}^{\alpha}\,dx  \notag \\
& +\gamma\int_{\Omega}\eta^4 (1+|Du|^2)^{\frac{\gamma-2}{2}%
}|Du|\sum_{i,s,\alpha}f_{\xi_{i}^{\alpha}x_s}(x,Du) u_{x_s}^{\alpha}\big(|Du|%
\big)_{x_i}\,dx  \notag \\
=:& I_{1}+I_2+I_3+I_4+I_5+I_6.  \label{3.5}
\end{align}

\textsc{Estimate of $I_1$}

\noindent By the use of Cauchy-Schwartz and Young's inequalities and by
virtue of the second inequality in \eqref{growth2(A2)--(A4)bis}, we can
estimate the integral $I_1$ as follows

\begin{eqnarray}  \label{I1}
|I_1| &\le & 4\int_{\Omega}\!\! (1+|Du|^2)^{\frac{\gamma}{2}}\Bigg\{%
\!\!\eta^2\sum_{i,j,s,\alpha,\beta}f_{\xi_{i}^{\alpha}\xi_{j}^{\beta}}(x,Du)%
\eta_{x_i} u_{x_s}^{\alpha} \eta_{x_j} u_{x_s}^{\beta}\Bigg\}^{\frac{1}{2}%
}\!\! \Bigg\{\eta^4\sum_{i,j,s,\alpha,\beta}f_{\xi_{i}^{\alpha}\xi_{j}^{%
\beta}}(x,Du) u_{x_sx_i}^{\alpha}u_{x_sx_j}^{\beta}\Bigg\}^{\frac{1}{2}} 
\notag \\
&\le & C(\varepsilon) \int_{\Omega}\eta^2|D\eta|^2 b(x)(1+|Du|^2)^{\frac{%
q+\gamma}{2}}\,dx  \notag \\
&&\qquad+ \varepsilon \int_{\Omega}\eta^4(1+|Du|^2)^{\frac{\gamma}{2}}
\sum_{i,j,s,\alpha,\beta}f_{\xi_{i}^{\alpha}\xi_{j}^{\beta}}(x,Du)
u_{x_sx_i}^{\alpha}u_{x_sx_j}^{\beta} \,dx,
\end{eqnarray}
where $\varepsilon >0$ will be chosen later.

\medbreak

\textsc{Estimate of $I_3$}

\noindent Since 
\begin{equation*}
f_{\xi_{i}^{\alpha}\xi_{j}^{\beta}}(x,\xi)=\left(\frac{g_{tt}(x,|\xi|)}{%
|\xi|^2} -\frac{g_{t}(x,|\xi|)}{|\xi|^3} \right)\xi_{i}^\alpha \xi_j^\beta+%
\frac{g_{t}(x,|\xi|)}{|\xi|}\delta_{\xi_i^\alpha \xi_j^\beta}
\end{equation*}
and 
\begin{equation}  \label{D(|Du|-1)}
(|Du|)_{x_i}= \frac{1}{|Du|}\sum_{\alpha,s}u_{x_ix_s}^{\alpha}u_{x_s}^{%
\alpha}
\end{equation}
then 
\begin{align}
& \sum_{i,j,s,\alpha,\beta}f_{\xi_{i}^{\alpha}\xi_{j}^{\beta}}(x,Du)
u_{x_s}^{\alpha}u_{x_sx_j}^{\beta} (|Du|)_{x_i}  \notag \\
= & \left(\frac{g_{tt}(x,|Du|)}{|Du|^2} -\frac{g_{t}(x,|Du|)}{|Du|^3}
\right)\sum_{i,j,s,\alpha,\beta} u_{x_s}^{\alpha}u_{x_sx_j}^{\beta}
u_{x_i}^{\alpha}u_{x_j}^{\beta} (|Du|)_{x_i}  \notag \\
& +\frac{g_{t}(x,|Du|)}{|Du|} \sum_{i,s,\alpha}
u_{x_s}^{\alpha}u_{x_sx_i}^{\alpha} (|Du|)_{x_i}  \notag \\
= & \left(\frac{g_{tt}(x,|Du|)}{|Du|} -\frac{g_{t}(x,|Du|)}{|Du|^2}
\right)\sum_{\alpha}\left(\sum_{i} u_{x_i}^{\alpha}(|Du|)_{x_i}\right)^2 
\notag \\
&+g_{t}(x,|Du|)|D(|Du|)|^2.  \label{3.7}
\end{align}
Thus, 
\begin{align*}
I_3=\gamma\int_{\Omega}\eta^4(1+|Du|^2)^{\frac{\gamma-2}{2}}|Du|&\Big\{ \Big(%
\frac{g_{tt}(x,|Du|)}{|Du|} -\frac{g_{t}(x,|Du|)}{|Du|^2} \Big)%
\sum_{\alpha}\left(\sum_{i} u_{x_i}^{\alpha}(|Du|)_{x_i}\right)^2 \\
& + g_{t}(x,|Du|)|D(|Du|)|^2\Big\}\,dx.
\end{align*}
Using the Cauchy-Schwartz inequality, i.e. 
\begin{equation*}
\sum_{\alpha}\left(\sum_{i} u_{x_i}^{\alpha}(|Du|)_{x_i}\right)^2\le
|Du|^2|D(|Du|)|^2
\end{equation*}
and observing that 
\begin{equation*}
g_{t}(x,|Du|)\ge 0,
\end{equation*}
we conclude 
\begin{equation}  \label{I3}
I_3\ge \gamma\int_{\Omega}\eta^4 (1+|Du|^2)^{\frac{\gamma-2}{2}}|Du|\frac{%
g_{tt}(x,|Du|)}{|Du|}\sum_{\alpha}\left(\sum_{i}
u_{x_i}^{\alpha}(|Du|)_{x_i}\right)^2\,dx\ge 0.
\end{equation}

\medbreak

\textsc{Estimate of $I_4$}

\noindent By using the last inequality in \eqref{growth2(A2)--(A4)bis} we
obtain 
\begin{eqnarray}
|I_4|&\le& 4\int_{\Omega}\eta^3 k(x)(1+|Du|^2)^{\frac{q-1+\gamma}{2}%
}\sum_{i,s,\alpha}|\eta_{x_i}u_{x_s}^{\alpha}|\,dx  \notag \\
&\le& 4\int_{\Omega}\eta^3|D\eta| k(x)(1+|Du|^2)^{\frac{q+\gamma}{2}}\,dx.
\label{I4}
\end{eqnarray}

\medbreak

\textsc{Estimate of $I_5$}

\noindent Using the last inequality in \eqref{growth2(A2)--(A4)bis} and
Young's inequality we have that 
\begin{eqnarray}  \label{I5}
|I_5| &\le& \int_{\Omega}\eta^4 k(x)(1+|Du|^2)^{\frac{q-1+\gamma}{2}%
}|D^2u|\,dx  \notag \\
&\le& \sigma \int_{\Omega}\eta^4a(x)(1+|Du|^2)^{\frac{p-2+\gamma}{2}%
}|D^2u|^2\,dx  \notag \\
&&+ C_\sigma\int_{\Omega}\eta^4\frac{k^2(x)}{a(x)}(1+|Du|^2)^{\frac{%
2q-p+\gamma}{2}}\,dx,
\end{eqnarray}
where $\sigma \in (0,1)$ will be chosen later and $a$ is the function
appearing in \eqref{growth2(A2)--(A4)bis}.

\medbreak

\textsc{Estimate of $I_6$}

\noindent Using the last inequality in \eqref{growth2(A2)--(A4)bis} and %
\eqref{D(|Du|-1)}, 
we get 
\begin{eqnarray}  \label{I6}
|I_6| &\le& \gamma\int_{\Omega} \eta^4k(x)(1+|Du|^2)^{\frac{q-1+\gamma}{2}}
|D(|Du|)|\,dx  \notag \\
& \le & \gamma\int_{\Omega}\eta^4(1+|Du|^2)^{\frac{q-1+\gamma}{2}%
}k(x)|D^2u|\,dx  \notag \\
&\le &\sigma\int_{\Omega}\eta^4a(x)(1+|Du|^2)^{\frac{p-2+\gamma}{2}%
}|D^2u|^2\,dx   \notag \\
&&+ C_\sigma \gamma^2\int_{\Omega}\eta^4\frac{k^2(x)}{a(x)}(1+|Du|^2)^{\frac{%
2q-p+\gamma}{2}}\,dx,
\end{eqnarray}
where we used Young's inequality again. \newline
Since the equality \eqref{3.5} can be written as follows 
\begin{equation*}
I_2+I_3=-I_1-I_4-I_5-I_6 \, ,
\end{equation*}
by virtue of \eqref{I3}, we get 
\begin{equation*}
I_2\le |I_1|+|I_4|+|I_5|+|I_6|\,
\end{equation*}
and therefore, recalling the estimates \eqref{I1}, \eqref{I4}, \eqref{I5}
and \eqref{I6}, we obtain 
\begin{eqnarray}  \label{stima2}
&&\int_{\Omega} \eta^4 (1+|Du|^2)^{\frac{\gamma}{2}}\sum_{i,j,s,\alpha,%
\beta}f_{\xi_{i}^{\alpha}\xi_{j}^{\beta}}(x,Du)
u_{x_sx_i}^{\alpha}u_{x_sx_j}^{\beta}\,dx  \notag \\
&\le& \varepsilon\int_{\Omega} \eta^4 (1+|Du|^2)^{\frac{\gamma}{2}%
}\sum_{i,j,s,\alpha,\beta}f_{\xi_{i}^{\alpha}\xi_{j}^{\beta}}(x,Du)
u_{x_sx_i}^{\alpha}u_{x_sx_j}^{\beta}\,dx  \notag \\
&&+4 \int_{\Omega}\eta^3|D\eta| k(x)(1+|Du|^2)^{\frac{q+\gamma}{2}}\,dx 
\notag \\
&&+2\sigma\int_{\Omega}\eta^4a(x)(1+|Du|^2)^{\frac{p-2+\gamma}{2}%
}|D^2u|^2\,dx  \notag \\
&&+C_\sigma(1+\gamma^2)\int_{\Omega}\eta^4\frac{k^2(x)}{a(x)}(1+|Du|^2)^{%
\frac{2q-p+\gamma}{2}}\,dx  \notag \\
&&+C_\varepsilon \int_{\Omega}\eta^2|D\eta|^2b(x)(1+|Du|^2)^{\frac{q+\gamma}{%
2}}\,dx.
\end{eqnarray}
Choosing $\varepsilon=\frac{1}{2}$, we can reabsorb the first integral in
the right hand side by the left hand side thus getting 
\begin{eqnarray}  \label{stima30}
&&\int_{\Omega} \eta^4 (1+|Du|^2)^{\frac{\gamma}{2}}\sum_{i,j,s,\alpha,%
\beta}f_{\xi_{i}^{\alpha}\xi_{j}^{\beta}}(x,Du)
u_{x_sx_i}^{\alpha}u_{x_sx_j}^{\beta}\,dx  \notag \\
&\le& 4\sigma\int_{\Omega}\eta^4a(x)(1+|Du|^2)^{\frac{p-2+\gamma}{2}%
}|D^2u|^2\,dx  \notag \\
&&+C\int_{\Omega}\eta^3|D\eta| k(x)(1+|Du|^2)^{\frac{q+\gamma}{2}%
}\,dx+C_\sigma(1+\gamma^2)\int_{\Omega}\eta^4\frac{k^2(x)}{a(x)}(1+|Du|^2)^{%
\frac{2q-p+\gamma}{2}}\,dx  \notag \\
&&+C \int_{\Omega}\eta^2|D\eta|^2b(x)(1+|Du|^2)^{\frac{q+\gamma}{2}}\,dx.
\end{eqnarray}
Now, using the ellipticity condition in \eqref{growth2(A2)--(A4)bis} to
estimate the left hand side of \eqref{stima30}, we get 
\begin{eqnarray*}
&&c_2 \int_{\Omega} \eta^4a(x)(1+|Du|^2)^{\frac{p-2+\gamma}{2}}|D^2u|^2\,dx
\\
&\le& 4\sigma\int_{\Omega}\eta^4a(x)(1+|Du|^2)^{\frac{p-2+\gamma}{2}%
}|D^2u|^2\,dx \\
&&+C\int_{\Omega}\eta^3|D\eta| k(x)(1+|Du|^2)^{\frac{q+\gamma}{2}%
}\,dx+C_\sigma(1+\gamma^2)\int_{\Omega}\eta^4\frac{k^2(x)}{a(x)}(1+|Du|^2)^{%
\frac{2q-p+\gamma}{2}}\,dx \\
&&+C \int_{\Omega}\eta^2|D\eta|^2b(x)(1+|Du|^2)^{\frac{q+\gamma}{2}}\,dx.
\end{eqnarray*}
We claim that $k \in L^{\frac{2s}{s-1}}_{\mathrm{loc}}(\Omega)$. Since by
assumption \eqref{kba2}, $k \in L^{r}_{\mathrm{loc}}(\Omega)$, we need to
prove that $\frac{2s}{s-1} \le r$ that is equivalent to $\frac{2}{r}+\frac{1%
}{s}\le 1$. This holds true, because by $\eqref{gap}$ and $q\ge p$ we get 
\begin{equation*}
\frac{s}{s+1}\left(1+\frac{1}{n}-\frac{1}{r}\right)>1\Leftrightarrow \frac{n%
}{r}+\frac{n}{s}<1
\end{equation*}
and we conclude, because $n\ge 2$. Therefore, since by the assumption $u\in
W^{1,\infty}_{\mathrm{loc}}(\Omega)\cap W^{2,\frac{2s}{s+1}}_{\mathrm{loc}%
}(\Omega)$, we can use Proposition \ref{weighted}, that implies that the
first integral in the right hand side of previous estimate is finite. By
choosing $\sigma=\frac{c_2}{8}$, we can reabsorb the first integral in the
right hand side by the left hand side thus getting 
\begin{eqnarray}  \label{stima3}
&&\int_{\Omega} \eta^4a(x)(1+|Du|^2)^{\frac{p-2+\gamma}{2}}|D^2u|^2\,dx 
\notag \\
&\le& C\int_{\Omega}\eta^3|D\eta| k(x)(1+|Du|^2)^{\frac{q+\gamma}{2}%
}\,dx+C(1+\gamma^2)\int_{\Omega}\eta^4\frac{k^2(x)}{a(x)}(1+|Du|^2)^{\frac{%
2q-p+\gamma}{2}}\,dx  \notag \\
&&+C \int_{\Omega}\eta^2|D\eta|^2b(x)(1+|Du|^2)^{\frac{q+\gamma}{2}}\,dx 
\notag \\
&\le& C\int_{\Omega}\left(\eta^2|D\eta|^2+|D\eta|^4\right) a(x)(1+|Du|^2)^{%
\frac{p+\gamma}{2}}\,dx  \notag \\
&&\quad +C(\gamma+1)^2\int_{\Omega}\eta^4\frac{k^2(x)+b^2(x)}{a(x)}%
(1+|Du|^2)^{\frac{2q-p+\gamma}{2}}\,dx,
\end{eqnarray}
where we used Young's inequality again. Now, we note that 
\begin{equation*}
\eta^4 a(x)\left|D\big((1+|Du|^2)^{\frac{p+\gamma}{4}}\big)\right|^2\le
c(p+\gamma)^2a(x)\eta^4(1+|Du|^2)^{\frac{p-2+\gamma}{2}}|D^2u|^2
\end{equation*}
and so fixing $\frac{R_0}{2}\le \rho<t^{\prime}<t<R<R_0$ with $R_0$ such
that $B_{R_0}\Subset\Omega$, and choosing $\eta\in C_0^\infty(B_t)$ a cut
off function between $B_{t^{\prime}}$ and $B_t$, by the assumption $%
a^{-1}\in L^{s}_{\mathrm{loc}}(\Omega)$ we can use Sobolev type inequality
of Lemma \ref{Sob} with $w=\eta^2(1+|Du|^2)^{\frac{p+\gamma}{4}}$, $%
\lambda=a $ and $p=2$, thus obtaining 
\begin{eqnarray*}
&&\left(\int_{B_t}\left(\eta^{2}(1+|Du|^2)^{\frac{p+\gamma}{4}%
}\right)^{\left(\frac{2s}{s+1}\right)^*}\,dx\right)^{\frac{2}{\left(\frac{2s%
}{s+1}\right)^*}}\le \frac{c(n)}{(t-t^{\prime})^2}\int_{B_t}a(x)(1+|Du|^2)^{%
\frac{p+\gamma}{2}}\,dx \\
&&\qquad\qquad+c(n)(p+\gamma)^2\int_{B_t} \eta^4 a\, (1+|Du|^2)^{\frac{%
p-2+\gamma}{2}}|D^2u|^2\,dx,
\end{eqnarray*}
with a constant $c(n)$ depending only on $n$.

Using \eqref{stima3} to estimate the last integral in previous inequality,
we obtain 
\begin{eqnarray*}
&&\left(\int_{B_t}\eta^{\frac{4ns}{n(s+1)-2s}}(1+|Du|^2)^{\frac{(p+\gamma)ns%
}{2(n(s+1)-2s)}}\,dx\right)^{\frac{n(s+1)-2s}{ns}} \\
&\le& c\left(\frac{(p+\gamma)^2}{(t-t^{\prime})^2}+\frac{(p+\gamma)^2}{%
(t-t^{\prime})^4}\right)\int_{B_t} a(x)(1+|Du|^2)^{\frac{p+\gamma}{2}}\,dx \\
&&\qquad+ c(p+\gamma)^4\int_{B_t}\frac{k^2(x)+b^2(x)}{a(x)}(1+|Du|^2)^{\frac{%
2q-p+\gamma}{2}}\,dx \\
&\le& c\left(\frac{(p+\gamma)^2}{(t-t^{\prime})^2}+\frac{(p+\gamma)^2}{%
(t-t^{\prime})^4}\right)\int_{B_t} a(x)(1+|Du|^2)^{\frac{p+\gamma}{2}}\,dx \\
&&\quad+ c(p+\gamma)^4\left(\int_{B_t}\frac{1}{a^{s}}\,dx\right)^{\frac{1}{s}%
}\left(\int_{B_t}(k^r+b^r)\,dx\right)^{\frac{2}{r}} \\
&&\qquad \times \left(\int_{B_t}(1+|Du|^2)^{\frac{(2q-p+\gamma)rs}{2(rs-2s-r)%
}}dx\right)^{\frac{rs-2s-r}{rs}},
\end{eqnarray*}
where we used assumptions \eqref{kba2} and H\"older's inequality with
exponents $s$, $\frac{r}{2}$ and $\frac{rs}{rs-2s-r}$.

Using the properties of $\eta $ we obtain 
\begin{eqnarray*}
&&\left( \int_{B_{t^{\prime }}}(1+|Du|^{2})^{\frac{(p+\gamma )ns}{%
2(n(s+1)-2s)}}\,dx\right) ^{\frac{n(s+1)-2s}{ns}} \\
&\leq &c(p+\gamma )^{4}\Vert a^{-1}\Vert _{L^{s}(B_{R_{0}})}\Vert k+b\Vert
_{L^{r}(B_{R_{0}})}^{2}\left( \int_{B_{t}}(1+|Du|^{2})^{\frac{(2q-p+\gamma
)rs}{2(rs-2s-r)}}dx\right) ^{\frac{rs-2s-r}{rs}} \\
&&\qquad +c\left( \frac{(p+\gamma )^{2}}{(t-t^{\prime })^{2}}+\frac{%
(p+\gamma )^{2}}{(t-t^{\prime })^{4}}\right) \int_{B_{t}}a(x)(1+|Du|^{2})^{%
\frac{p+\gamma }{2}}\,dx \\
&\leq &c(p+\gamma )^{4}\Vert a^{-1}\Vert _{L^{s}(B_{R_{0}})}\Vert k+b\Vert
_{L^{r}(B_{R_{0}})}^{2}\left( \int_{B_{t}}(1+|Du|^{2})^{\frac{(2q-p+\gamma
)rs}{2(rs-2s-r)}}dx\right) ^{\frac{rs-2s-r}{rs}} \\
&&\qquad +c\left( \frac{(p+\gamma )^{2}}{(t-t^{\prime })^{2}}+\frac{%
(p+\gamma )^{2}}{(t-t^{\prime })^{4}}\right) \Vert a\Vert _{L^{\frac{rs}{2s+r%
}}(B_{R_{0}})}\left( \int_{B_{t}}(1+|Du|^{2})^{\frac{(p+\gamma )rs}{%
2(rs-2s-r)}}\,dx\right) ^{\frac{rs-2s-r}{rs}},
\end{eqnarray*}%
where we used the assumption $a\in L_{\mathrm{loc}}^{\frac{rs}{2s+r}}(\Omega
)$. Setting 
\begin{equation}
\mathcal{K}_{R_{0}}=1+\Vert a^{-1}\Vert _{L^{s}(B_{R_{0}})}\Vert k+b\Vert
_{L^{r}(B_{R_{0}})}^{2}+\Vert a\Vert _{L^{\frac{rs}{2s+r}}}(B_{R_{0}})
\label{kappaR0}
\end{equation}%
and assuming without loss of generality that $t-t^{\prime }<1$, we can write
the previous estimate as follows 
\begin{align*}
& \left( \int_{B_{t^{\prime }}}(1+|Du|^{2})^{\frac{(p+\gamma )ns}{%
2(n(s+1)-2s)}}\,dx\right) ^{\frac{n(s+1)-2s}{ns}} \\
& \leq c(p+\gamma )^{4}\mathcal{K}_{R_{0}}\left( \int_{B_{t}}(1+|Du|^{2})^{%
\frac{(2q-p+\gamma )rs}{2(rs-2s-r)}}dx\right) ^{\frac{rs-2s-r}{rs}} \\
& \qquad +c(p+\gamma )^{2}\frac{\mathcal{K}_{R_{0}}}{(t-t^{\prime })^{4}}%
\left( \int_{B_{t}}(1+|Du|^{2})^{\frac{(p+\gamma )rs}{2(rs-2s-r)}%
}\,dx\right) ^{\frac{rs-2s-r}{rs}},
\end{align*}%
and, using the a-priori assumption $u\in W_{\mathrm{loc}}^{1,\infty }(\Omega
)$, we get 
\begin{align}
& \left( \int_{B_{t^{\prime }}}(1+|Du|^{2})^{\frac{(p+\gamma )ns}{%
2(n(s+1)-2s)}}\,dx\right) ^{\frac{n(s+1)-2s}{ns}}  \notag \\
& \leq c(p+\gamma )^{4}\mathcal{K}_{R_{0}}\left( \Vert Du\Vert _{L^{\infty
}(B_{R})}^{2(q-p)}+\frac{1}{(t-t^{\prime })^{4}}\right) \left(
\int_{B_{t}}(1+|Du|^{2})^{\frac{(p+\gamma )rs}{2(rs-2s-r)}}dx\right)^{\frac{%
rs-2s-r}{rs}}.  \label{stima3nu}
\end{align}

Setting now 
\begin{equation*}
m=\frac{rs}{rs-2s-r}
\end{equation*}
and noting that 
\begin{equation*}
\frac{ns}{n(s+1)-2s}=\frac{1}{2}\left(\frac{2s}{s+1}\right)^*=:\frac{2_s^*}{2%
}
\end{equation*}
we can write \eqref{stima3nu} as follows 
\begin{eqnarray}  \label{ultima}
&&\left(\int_{B_{t^{\prime}}}\left((1+|Du|^2)^{\frac{(p+\gamma)m}{2}%
}\right)^{\frac{2_s^*}{2m}}\,dx\right)^{\frac{2m}{2_s^*}}  \notag \\
&\le& c(p+\gamma)^{4m}\mathcal{K}^m_{R_0}\frac{\|Du\|^{2(q-p)m}_{L^%
\infty(B_{R})}}{(t-t^{\prime})^{4m}}\int_{B_t}(1+|Du|^2)^{\frac{(p+\gamma)m}{%
2}}dx,
\end{eqnarray}
where, without loss of generality, we supposed $\|Du\|^{2(q-p)m}_{L^%
\infty(B_{R})}\ge 1$. Define now the decreasing sequence of radii by setting 
\begin{equation*}
\rho_i= \rho+\frac{ R- \rho}{2^i}
\end{equation*}
and the increasing sequence of exponents 
\begin{equation*}
p_0=pm\qquad\qquad {p_{i}}={p_{i-1}}\left(\frac{2_s^*}{2m}\right)=p_0\left(%
\frac{2_s^*}{2m}\right)^{i}
\end{equation*}
As we will prove (see \eqref{gamma0} below) the right hand side of %
\eqref{ultima} is finite for $\gamma=0$ . Then for every $%
\rho<\rho_{i+1}<\rho_{i}< R$, we may iterate it on the concentric balls $%
B_{\rho_i}$ with exponents $p_i$, thus obtaining%
\begin{eqnarray}  \label{stimapreitepassoi}
&&\left ( \int_{B_{\rho_{i+1}}} (1+ |Du|^2)^{\frac{p_{i+1}}{2}}\, dx\right
)^{\frac{1}{p_{i+1}}}  \notag \\
&\le& \displaystyle{\prod_{j=0}^{i}}\left(C^m \mathcal{K}^m_{R_0}{\frac{
p_j^{4m}\|Du\|^{2m(q-p)}_{L^\infty(B_R)}}{(\rho_{j} - \rho_{j+1})^{4m}}}%
\right)^{\frac{1}{p_j}}\left(\int_{B_{R}}(1+|Du|^2)^{\frac{p_0}{2}} \,
dx\right)^{\frac{1}{p_0}}  \notag \\
&=& \displaystyle{\prod_{j=0}^{i}}\left(C^m \mathcal{K}^m_{R_0}{\frac{4^{jm}
p_j^{4m}\|Du\|^{2(q-p)m}_{L^\infty(B_{R})}}{(R - \rho)^{4m}}}\right)^{\frac{1%
}{p_j}}\left(\int_{B_{R}}(1+|Du|^2)^{\frac{p_0}{2}} \, dx\right)^{\frac{1}{%
p_0}}  \notag \\
&=& \displaystyle{\prod_{j=0}^{i}}\left(4^{jm} p_j^{4m} \right)^{\frac{1}{p_j%
}}\displaystyle{\prod_{j=0}^{i}}\left({\frac{C^m \mathcal{K}^m_{R_0}
\|Du\|^{2(q-p)m}_{L^\infty(B_{R})}}{(R - \rho)^{4m}}}\right)^{\frac{1}{p_j}%
}\left(\int_{B_{R}}(1+|Du|^2)^{\frac{p_0}{2}} \, dx\right)^{\frac{1}{p_0}}.
\end{eqnarray}
Since 
\begin{equation*}
\displaystyle{\prod_{j=0}^{i}}\left(4^{jm} p_j^{4m}\right)^{\frac{1}{p_j}%
}=\exp\left(\sum_{j=0}^i \frac{1}{p_j}\log(4^{jm} p_j^{4m})\right) \le
\exp\left(\sum_{j=0}^{+\infty} \frac{1}{p_j}\log(4^{jm} p_j^{4m})\right) \le
c(n,r)
\end{equation*}
and 
\begin{eqnarray*}
&&\displaystyle{\prod_{j=0}^{i}}\left(\frac{C^m \mathcal{K}%
^m_{R_0}\|Du\|^{2(q-p)m}_{L^\infty(B_t)}}{{(R - \rho)^{4m}}}\right)^{\frac{1%
}{p_j}}=\left(\frac{C^m \mathcal{K}^m_{R_0}\|Du\|^{2(q-p)m}_{L^\infty(B_R)}}{%
{(R - \rho)^{4m}}}\right)^{\sum_{j=0}^{i}\frac{1}{p_j}} \\
&\le& \left(\frac{C^m \mathcal{K}^m_{R_0} \|Du\|^{2(q-p)m}_{L^\infty(B_R)}}{{%
(R - \rho)^{4m}}}\right)^{\sum_{j=0}^{+\infty}\frac{1}{p_j}}=\left(\frac{C 
\mathcal{K}_{R_0}\|Du\|^{2(q-p)}_{L^\infty(B_R)}}{{(R - \rho)^{4}}}\right)^{%
\frac{2^*_s}{p(2^*_s-2m)}}
\end{eqnarray*}
we can let $i\to \infty$ in \eqref{stimapreitepassoi} thus getting 
\begin{eqnarray*}
\|Du\|_{L^{\infty}(B_{\rho})} \le C(n,r,p) \left(\frac{\mathcal{K}_{R_0}}{{%
(R - \rho)^{4}}}\right)^{\frac{2^*_s}{p(2^*_s-2m)}}\|Du\|^{\frac{2(q-p)2^*_s%
}{p(2^*_s-2m)}}_{L^\infty(B_R)}\left(\int_{B_R}(1+|Du|^2)^{\frac{pm}{2}} \,
dx\right)^{\frac{1}{pm}},
\end{eqnarray*}
where we used that $p_0=pm$. Since assumption \eqref{gap} implies that 
\begin{equation}  \label{e:2*s}
\frac{2(q-p)2^*_s}{p(2^*_s-2m)}<1,
\end{equation}
we can use Young's inequality with exponents 
\begin{equation*}
\frac{p(2^*_s-2m)}{2(q-p)2^*_s}>1\qquad \text{and}\qquad \frac{p(2^*_s-2m)}{%
p(2^*_s-2m)-2(q-p)2^*_s}
\end{equation*}
to deduce that 
\begin{eqnarray}  \label{e:inftym}
\|Du\|_{L^{\infty}(B_{\rho})} \le \frac{1}{2}\|Du\|_{L^\infty(B_R)}+
C(n,r,p,s) \left(\frac{\mathcal{K}_{R_0}}{{(R - \rho)^{4}}}%
\right)^{\vartheta}\left(\int_{B_R}(1+|Du|^2)^{\frac{pm}{2}} \,
dx\right)^{\varsigma},
\end{eqnarray}
with $\vartheta=\vartheta(p,q,n,s)$ and $\varsigma=\varsigma(p,q,n,s)$.

We now estimate the last integral. By definition of $m$ and by the
assumption on $s$, i.e., $s>\frac{nr}{r-n}$, we get 
\begin{equation*}
m<\frac{ns}{n(s+1)-2s}.
\end{equation*}
Thus, by H\"older's inequality, 
\begin{equation}  \label{e:stimam}
\int_{B_{R}}(1+|Du|^2)^{\frac{pm}{2}}\,dx\le c(R_0,n,r,s)
\left(\int_{B_{R}}(1+|Du|^2)^{\frac{pns}{2(n(s+1)-2s)}}\,dx\right)^{\frac{%
r(n(s+1)-2s)}{n(rs-2s-r)}}.
\end{equation}
This last integral can be estimated by using \eqref{stima3nu} with $\gamma=0$%
. Indeed, let us re-define $t^{\prime},t$ and $\eta$ as follows: consider $%
R\le t^{\prime}<t\le 2R-\rho \le R_0$ and $\eta$ a cut off function, $%
\eta\equiv 1$ on $B_{t^{\prime}}$ and $\mathrm{supp\,}\eta\subset B_{t}$. By %
\eqref{stima3nu} with $\gamma=0$, 
\begin{eqnarray}
&&\left(\int_{B_{t^{\prime}}}(1+|Du|^2)^{\frac{pns}{2(n(s+1)-2s)}%
}\,dx\right)^{\frac{n(s+1)-2s}{ns}}  \notag \\
&\le& c\left(\frac{p^2}{(t-t^{\prime})^2}+\frac{p^2}{(t-t^{\prime})^4}%
\right)\int_{B_{R_0}} a(x)(1+|Du|^2)^{\frac{p}{2}}\,dx  \notag \\
&&\qquad+cp^4\mathcal{K}_{B_{R_0}}\left(\int_{B_{R_0}}(k^r+b^r)\,dx\right)^{%
\frac{2}{r}}  \notag \\
&&\qquad\qquad\qquad\qquad\qquad\times \left(\int_{B_t}(1+|Du|^2)^{\frac{%
(2q-p)rs}{2(rs-2s-r)}}dx\right)^{\frac{rs-2s-r}{rs}}.  \label{stima3zero}
\end{eqnarray}
If we denote 
\begin{equation*}
\tau:= \frac{(2q-p)rs}{rs-2s-r}, \quad \tau_1:=\frac{nps}{n(s+1)-2s},\quad
\tau_2:=\frac{ps}{s+1},
\end{equation*}
by \eqref{gap} and $s>\frac{rn}{r-n}$, we get 
\begin{equation*}
\frac{\tau}{\tau_1}< 1<\frac{\tau}{\tau_2}.
\end{equation*}

Therefore there exists $\theta\in (0,1)$ such that 
\begin{equation*}
1=\theta\frac{\tau}{\tau_1}+(1-\theta)\frac{\tau}{\tau_2}.
\end{equation*}
The precise value of $\theta$ is 
\begin{equation}  \label{e:theta}
\theta=\frac{ns(qr-pr+p)+qrn}{rs(2q-p)}.
\end{equation}
By H\"older's inequality with exponents $\frac{\tau_1}{\theta\tau}$ and $%
\frac{\tau_2}{(1-\theta)\tau}$ we get 
\begin{align*}
& \left(\int_{B_{t}}(1+|Du|^2)^{\frac{(2q-p)rs}{2(rs-2s-r)}}dx\right)^{\frac{%
rs-2s-r}{rs}} = \left(\int_{B_{t}}(1+|Du|^2)^{\theta\frac{\tau}{2}+(1-\theta)%
\frac{\tau}{2}}dx\right)^{\frac{2q-p}{\tau}} \\
&\le \left(\int_{B_{t}}(1+|Du|^2)^{\frac{\tau_1}{2}}dx\right)^{\frac{%
(2q-p)\theta}{\tau_1}} \left(\int_{B_{t}}(1+|Du|^2)^{\frac{\tau_2}{2}%
}dx\right)^{\frac{(2q-p)(1-\theta)}{\tau_2}}.
\end{align*}

Hence, we can use the inequality above to estimate the last integral of %
\eqref{stima3zero} to deduce that 
\begin{eqnarray}  \label{stima3zero1}
&&\left(\int_{B_{t^{\prime}}} (1+|Du|^2)^{\frac{pns}{2(n(s+1)-2s)}%
}\,dx\right)^{\frac{n(s+1)-2s}{ns}}  \notag \\
&\le& c\left(\frac{p^2}{(t-t^{\prime})^2}+\frac{p^2}{(t-t^{\prime})^4}%
\right)\int_{B_{R_0}} a(x)(1+|Du|^2)^{\frac{p}{2}}\,dx  \notag \\
&& +C \mathcal{K}_{B_{R_0}}\left(\int_{B_{t}}(1+|Du|^2)^{\frac{ps}{2(s+1)}%
}\,dx\right)^{\frac{(1-\theta)(2q-p)(s+1)}{ps}}  \notag \\
&&\qquad\times \left(\int_{B_{t}}(1+|Du|^2)^{\frac{pns}{2(n(s+1)-2s)}%
}\,dx\right)^{\frac{\theta(ns+n-2s)(2q-p)}{nps}}.
\end{eqnarray}
Note that, again by \eqref{gap} and \eqref{e:theta}, we have 
\begin{equation*}
\frac{\theta(2q-p)}{p}<1.
\end{equation*}
%
%
We can use Young's inequality in the last term of \eqref{stima3zero1} with
exponents $\frac{p}{p-\theta(2q-p)}$ and $\frac{p}{\theta(2q-p)}$ to obtain
that for every $\sigma<1$ 
\begin{align*}
& \left(\int_{B_{t^{\prime}}}(1+|Du|^2)^{\frac{pns}{2(n(s+1)-2s)}%
}\,dx\right)^{\frac{n(s+1)-2s}{ns}} \\
\le &C\left(\frac{p^2}{(t-t^{\prime})^2}+\frac{p^2}{(t-t^{\prime})^4}%
\right)\int_{B_{R_0}} a(x)(1+|Du|^2)^{\frac{p}{2}}\,dx \\
& +C_\sigma\mathcal{K}_{B_{R_0}}^{\frac{p}{p-\theta(2q-p)}}
\left(\int_{B_{t}}(1+|Du|^2)^{\frac{ps}{2(s+1)}}\,dx\right)^{ \frac{%
(1-\theta)(2q-p)}{p-\theta (2q-p)}\frac{s+1}{s} } \\
&\qquad+\sigma \left(\int_{B_{t}}(1+|Du|^2)^{\frac{pns}{2(n(s+1)-2s)}%
}\,dx\right)^{\frac{n(s+1)-2s}{ns}}.
\end{align*}
By applying Lemma \ref{iter}, and noting that $2R-\rho-R=R-\rho$, we
conclude that 
\begin{eqnarray}  \label{highint0}
&& \left(\int_{B_{R}}(1+|Du|^2)^{\frac{pns}{2(n(s+1)-2s)}}\,dx\right)^{\frac{%
n(s+1)-2s}{ns}}  \notag \\
&&\le C\left(\frac{p^2}{(R-\rho)^2}+\frac{p^2}{(R-\rho)^4}%
\right)\int_{B_{R_0}} a(x)(1+|Du|^2)^{\frac{p}{2}}\,dx  \notag \\
&& +C\mathcal{K}_{B_{R_0}}^{\frac{p}{p-\theta(2q-p)}}
\left(\int_{B_{R_0}}(1+|Du|^2)^{\frac{ps}{2(s+1)}}\,dx\right)^{\frac{%
(1-\theta)(2q-p)}{p-\theta (2q-p)}\frac{s+1}{s}}.
\end{eqnarray}

Collecting \eqref{e:stimam} and \eqref{highint0} we obtain 
\begin{eqnarray}  \label{gamma0}
&&\int_{B_{R}}(1+|Du|^2)^{\frac{pm}{2}}\,dx  \notag \\
\le && C\left(\frac{p^2}{(R-\rho)^2}+\frac{p^2}{(R-\rho)^4}%
\right)\int_{B_{R_0}} a(x)(1+|Du|^2)^{\frac{p}{2}}\,dx  \notag \\
&& +C\mathcal{K}_{B_{R_0}}^{\frac{p}{p-\theta(2q-p)}}
\left(\int_{B_{R_0}}(1+|Du|^2)^{\frac{ps}{2(s+1)}}\,dx\right)^{\frac{%
(1-\theta)(2q-p)}{p-\theta (2q-p)}\frac{s+1}{s}}.
\end{eqnarray}
Notice that the right hand side is finite, because $u$ is a local minimizer
and \eqref{(A2)--(A4)-I} and \eqref{intu} hold. This inequality, together
with \eqref{e:inftym}, implies 
\begin{align*}
\Vert Du\Vert _{L^{\infty }(B_{\rho })}\leq& \frac{1}{2}\Vert Du\Vert
_{L^{\infty }(B_{R})}+C\left( \frac{\mathcal{K}_{B_{R_{0}}}}{(R-\rho )^{8}}%
\right) ^{\theta }\left( \int_{B_{R_{0}}}a(x)(1+|Du|^{2})^{\frac{p}{2}%
}\,dx\right) ^{\tilde{\theta}} \\
&+C\left( \frac{\mathcal{K}_{B_{R_{0}}}}{(R-\rho )^{8}}\right) ^{\theta
}\left( \int_{B_{R_{0}}}(1+|Du|^{2})^{\frac{ps}{2(s+1)}}\,dx\right) ^{\tilde{%
\varsigma}}
\end{align*}%
with the constant $C$ depending on the data. Applying Lemma \ref{iter} we
conclude the proof of estimate \eqref{stimafin}. Now, we write the estimate %
\eqref{stima3} for $\gamma =0$ and for a cut off function $\eta \in
C_{0}^{\infty }(B_{\frac{R}{2}})$, $\eta =1$ on $B_{\rho }$ for some $\rho <%
\frac{R}{2}$. This yields 
\begin{eqnarray*}
&&\int_{B_{\rho }}a(x)(1+|Du|^{2})^{\frac{p-2}{2}}|D^{2}u|^{2}\,dx \\
&\leq &C(R)\int_{B_{\frac{R}{2}}}a(x)(1+|Du|^{2})^{\frac{p}{2}%
}\,dx+C\int_{B_{\frac{R}{2}}}\frac{k^{2}(x)+b^{2}(x)}{a(x)}(1+|Du|^{2})^{%
\frac{2q-p}{2}}\,dx \\
&\leq &C(R)\int_{B_{\frac{R}{2}}}f(x,Du)\,dx+C\Vert 1+|Du|\Vert _{L^{\infty
}(B_{\frac{R}{2}})}^{2q-p}\int_{B_{\frac{R}{2}}}\frac{k^{2}(x)+b^{2}(x)}{a(x)%
}\,dx \\
&\leq &C(R)\int_{B_{\frac{R}{2}}}f(x,Du)\,dx \\
&&\quad +C(R)\Vert 1+|Du|\Vert _{L^{\infty }(B_{\frac{R}{2}})}^{2q-p}\left(
\int_{B_{\frac{R}{2}}}(k^{r}(x)+b^{r}(x))\,dx\right) ^{\frac{2}{r}}\left(
\int_{B_{\frac{R}{2}}}\frac{1}{a^{s}(x)}\,dx\right) ^{\frac{1}{s}},
\end{eqnarray*}%
where we used H\"{o}lder's inequality, since $\frac{1}{s}+\frac{2}{r}<1$ by
assumptions. Using \eqref{stimafin} to estimate the $L^{\infty }$ norm of $%
|Du|$ and recalling the definition of $\mathcal{K}_{R_{0}}$ at %
\eqref{kappaR0}, we get 
\begin{equation*}
\int_{B_{\rho }}a(1+|Du|^{2})^{\frac{p-2}{2}}|D^{2}u|^{2}\,dx\leq c\left(
\int_{B_{R}}1+f(x,Du)\,dx\right) ^{\tilde{\varrho}},
\end{equation*}%
i.e. \eqref{stimafin2}, with $c$ depending on $p,r,s,n,\rho ,R,\mathcal{\ K}%
_{R_{0}}$.
\end{proof}

\bigskip

\bigskip

\section{An auxiliary functional: higher differentiability estimate}

\label{s:highdiff}

Consider the functional 
\begin{equation*}
H(v)=\int_{\Omega }h(x,Dv)\,dx
\end{equation*}%
where $\Omega \subset \mathbb{R}^{n}$, $n\geq 2$, is a Sobolev map and $%
h:\Omega \times \mathbb{R}^{N\times n}\rightarrow \lbrack 0,+\infty )$ is a
Carath\'{e}odory function, convex and of class $C^{2}$ with respect to the
second variable. We assume that there exists $\tilde{h}:\Omega \times
\lbrack 0,+\infty )\rightarrow \lbrack 0,+\infty )$, increasing in the last
variable such that 
\begin{equation}
h(x,\xi )=\tilde{h}(x,|\xi |).  \label{(A1g)}
\end{equation}%
Moreover assume that there exist $p,q$, $1<p<q$, and constant $\ell ,\nu
,L_{1},L_{2}$ such that 
\begin{equation}
\ell (1+|\xi |^{2})^{\frac{p}{2}}\leq h(x,\xi )\leq L_{1}\,(1+|\xi |^{2})^{%
\frac{q}{2}}  \label{growthb}
\end{equation}%
%
%
for a.e. $x\in \Omega $ and for every $\xi \in \mathbb{R}^{N\times n}$. We
assume that $h$ is of class $C^{2}$ with respect to the $\xi -$variable, and
that the following conditions hold 
\begin{equation}
\nu \,(1+|\xi |^{2})^{\frac{p-2}{2}}|\lambda |^{2}\leq \langle h_{\xi \xi
}(x,\xi )\lambda ,\lambda \rangle \leq L_{2}\,(1+|\xi |^{2})^{\frac{q-2}{2}%
}|\lambda |^{2}  \label{(A2g)}
\end{equation}%
for a.e. $x\in \Omega $ and for every $\xi ,\lambda \in \mathbb{R}^{N\times
n}$. Moreover, we assume that there exists a non-negative function $k\in L_{%
\mathrm{loc}}^{r}(\Omega )$ such that 
\begin{equation}
|D_{\xi x}h(x,\xi )|\leq k(x)(1+|\xi |^{2})^{\frac{q-1}{2}}  \label{(A4g)}
\end{equation}%
for a.e. $x\in \Omega $ and for every $\xi \in \mathbb{R}^{N\times n}$.

The following is a higher differentiability result for minimizers of $H$,
that, by the result in \cite{EMM1}, are locally Lipschitz continuos.

\begin{theorem}
\label{highdiff} Let $v\in W^{1,\infty}_\mathrm{loc}(\Omega)$ be a local
minimizer of the functional $H$. Assume \eqref{(A1g)}--\eqref{(A4g)} for a
couple of exponents $p,q$ such that 
\begin{equation}  \label{gaph}
\frac{q}{p}<1+\frac{1}{n}-\frac{1}{r}.
\end{equation}
Then $u\in W^{2,2}_{\mathrm{loc}}(\Omega)$ and the following estimate holds 
\begin{eqnarray*}
&&\int_{B_\rho} |D V_p(Du)|^2\,dx \\
&\le& \frac{c}{(R-\rho)^2} \|1+|Du|\|^{2q-p}_{\infty}\int_{B_{2R}}|D
u|^2\,dx+c\|1+|Du|\|^{2q-p}_{\infty}\int_{B_R} k^2(x)\,dx \\
&&+ c\|1+|Du|\|^{q-1}_{\infty}\left(\int_{B_R} k^{\frac{p}{p-1}%
}(x)\,dx\right)^{\frac{p-1}{p}}\left(\int_{B_{2R}}|D u|^p\,dx\right)^{\frac{1%
}{p}}
\end{eqnarray*}
for every ball $B_\rho\subset B_R\subset B_{2R}\Subset\Omega.$
\end{theorem}
\begin{proof}
Since $v$ is a local minimizer of the functional $H$, then $v$ satisfies the
Euler's system 
\begin{equation*}
\int_{\Omega }\sum_{i,\alpha }h_{\xi _{i}^{\alpha }}(x,Du)\varphi
_{x_{i}}^{\alpha }(x)\,dx=0\qquad \forall \varphi \in C_{0}^{\infty }(\Omega
;\mathbb{R}^{N}).
\end{equation*}%
Let $B_{2R}\Subset \Omega $ and let $\eta \in C_{0}^{\infty }(B_{R})$ be a
cut off function between $B_{\rho }$ and $B_{R}$ for some $\rho <R$. Fixed $%
1\leq s\leq n$, and denoted $e_{s}$ is the unit vector in the $x_{s}$
direction, consider the finite differential operator $\tau _{s,h}$, see %
\eqref{e:rappincr}, from now on simply denoted $\tau _{h}$. Choosing $%
\varphi =\tau _{-h}(\eta ^{2}\tau _{h}u)$ as test function in the Euler's
system, we get, by properties (i) and (ii) of $\tau _{h}$, 
\begin{equation*}
\int_{\Omega }\sum_{i,\alpha }\tau _{h}(h_{\xi _{i}^{\alpha
}}(x,Du))D_{x_{i}}(\eta ^{2}\tau _{h}u^{\alpha })\,dx=0
\end{equation*}%
and so 
\begin{equation*}
\int_{\Omega }\sum_{i,\alpha }\tau _{h}(h_{\xi _{i}^{\alpha }}(x,Du))(2\eta
\eta _{x_{i}}\tau _{h}u^{\alpha }+\eta ^{2}\tau _{h}u_{x_{i}}^{\alpha
})\,dx=0.
\end{equation*}%
Exploiting the definition of $\tau _{h}$, we get 
\begin{equation*}
\int_{\Omega }\big(h_{\xi _{i}^{\alpha }}(x+he_{s},Du(x+he_{s}))-h_{\xi
_{i}^{\alpha }}(x,Du(x))\big)(2\eta \eta _{x_{i}}\tau _{h}u^{\alpha }+\eta
^{2}\tau _{h}u_{x_{i}}^{\alpha })\,dx=0
\end{equation*}%
i.e. 
\begin{align*}
&\int_{\Omega }\eta ^{2}[h_{\xi _{i}^{\alpha
}}(x+he_{s},Du(x+he_{s}))-h_{\xi _{i}^{\alpha }}(x+he_{s},Du(x))]\tau
_{h}u_{x_{i}}^{\alpha }\,dx \\
=&\int_{\Omega }\eta ^{2}[h_{\xi _{i}^{\alpha }}(x,Du(x))-h_{\xi
_{i}^{\alpha }}(x+he_{s},Du(x))]\tau _{h}u_{x_{i}}^{\alpha }\,dx \\
& -2\int_{\Omega }\eta \lbrack h_{\xi _{i}^{\alpha
}}(x+he_{s},Du(x+he_{s}))-h_{\xi _{i}^{\alpha }}(x+he_{s},Du(x))]\eta
_{x_{i}}\tau _{h}u^{\alpha } \\
&+2\int_{\Omega }\eta \lbrack h_{\xi _{i}^{\alpha }}(x,Du(x))-h_{\xi
_{i}^{\alpha }}(x+he_{s},Du(x))]\eta _{x_{i}}\tau _{h}u^{\alpha }.
\end{align*}%
We can write previous equality as follows 
\begin{align*}
I_{0}=:&\int_{\Omega }\eta ^{2}\int_{0}^{1}\sum_{i,j,\alpha ,\beta }h_{\xi
_{i}^{\alpha }\xi _{j}^{\beta }}(x+he_{s},Du(x)+\sigma \tau
_{h}Du(x))d\sigma \tau _{h}u_{x_{i}}^{\alpha }\tau _{h}u_{x_{j}}^{\beta }\,dx
\\
=& h\int_{\Omega }\eta ^{2}\int_{0}^{1}\sum_{i,\alpha }h_{x_{s}\xi
_{i}^{\alpha }}(x+\sigma he_{s},Du(x))d\sigma \tau _{h}u_{x_{i}}^{\alpha
}\,dx \\
& -2\int_{\Omega }\eta \int_{0}^{1}\sum_{i,j,\alpha ,\beta }h_{\xi
_{i}^{\alpha }\xi _{j}^{\beta }}(x+he_{s},Du(x)+\sigma \tau
_{h}Du(x))d\sigma \eta _{x_{i}}\tau _{h}u^{\alpha }\tau _{h}u_{x_{j}}^{\beta
} \\
& +2h\int_{\Omega }\eta \int_{0}^{1}\sum_{i,\alpha }h_{x_{s}\xi _{i}^{\alpha
}}(x+\sigma he_{s},Du(x))d\sigma \eta _{x_{i}}\tau _{h}u^{\alpha } \\
=: &I_{1}+I_{2}+I_{3},
\end{align*}%
that implies 
\begin{equation}
I_{0}\leq |I_{1}|+|I_{2}|+|I_{3}|  \label{zero}
\end{equation}%
The ellipticity assumption \eqref{(A2g)} yields 
\begin{equation}
I_{0}\geq \nu \int_{\Omega }\eta ^{2}(1+|Du(x)|^{2}+|Du(x+he_{s})|^{2})^{%
\frac{p-2}{2}}|\tau _{h}Du|^{2}\,dx  \label{hd0}
\end{equation}%
By assumption \eqref{(A4g)} we get 
\begin{eqnarray}
|I_1|&\le &|h|\int_{\Omega}\eta^2 \int_0^1 k(x+\sigma he_s)d\sigma
(1+|Du(x)|^2)^{\frac{q-1}{2}}|\tau_h Du|\,dx  \notag \\
&\le& |h|\|1+|Du|\|^{\frac{2q-p}{2}}_{\infty}\int_{\Omega}\eta^2 \int_0^1
k(x+\sigma he_s)d\sigma (1+|Du(x)|^{2})^{\frac{p-2}{4}}|\tau_h Du|\,dx 
\notag \\
&\le&\frac{\nu}{4}\int_{\Omega}\eta^2 (1+|Du(x)|^{2})^{\frac{p-2}{2}}|\tau_h
Du|^2\,dx  \notag \\
&&+ c_\nu|h|^2\|1+|Du|\|^{2q-p}_{\infty}\int_{\Omega}\eta^2 \left(\int_0^1
k(x+\sigma he_s)d\sigma\right)^2\,dx,  \label{hd1}
\end{eqnarray}
where in the last line we used Young's inequality. The right inequality in %
\eqref{(A2g)} yields 
\begin{eqnarray}
|I_2|&\le &c(L_2)\int_{\Omega}\eta|D\eta| (1+|Du(x)|^2+|Du(x+he_s)|^2)^{%
\frac{q-2}{2}}|\tau_h Du||\tau_h u|\,dx  \notag \\
&\le& c(L_2)\|1+|Du|\|^{\frac{2q-p}{2}}_{\infty}\int_{\Omega}\eta|D\eta|
(1+|Du(x)|^{2}+|Du(x+he_s)|^2)^{\frac{p-2}{4}}|\tau_h Du||\tau_h u|\,dx 
\notag \\
&\le&\frac{\nu}{4}\int_{\Omega}\eta^2 (1+|Du(x)|^{2}+|Du(x+he_s)|^2)^{\frac{%
p-2}{2}}|\tau_h Du|^2\,dx  \notag \\
&&+c_{\nu,L_2}\|1+|Du|\|^{2q-p}_{\infty}\int_{\Omega}|D\eta|^2|\tau_h
u|^2\,dx.  \label{hd2}
\end{eqnarray}%
Finally, using again assumption \eqref{(A4g)} and H\"{o}lder's inequality,
we obtain 
\begin{eqnarray}
|I_3|&\le &2|h|\int_{\Omega}\eta|D\eta| \int_0^1 k(x+\sigma he_s)d\sigma
(1+|Du(x)|^2)^{\frac{q-1}{2}}|\tau_h u|\,dx  \notag \\
&\le& |h|\|1+|Du|\|^{q-1}_{\infty}\int_{\Omega}\eta |D\eta| \int_0^1
k(x+\sigma he_s)d\sigma |\tau_h u|\,dx  \notag \\
&\le& |h| \|1+|Du|\|^{q-1}_{\infty}\left(\int_{\Omega}\eta|D\eta|
\left(\int_0^1 k(x+\sigma he_s)d\sigma\right)^{\frac{p}{p-1}}\,dx\right)^{%
\frac{p-1}{p}}  \notag \\
&&\qquad\times\left(\int_{\Omega}\eta|D\eta||\tau_h u|^p\,dx\right)^{\frac{1%
}{p}}.  \label{hd3}
\end{eqnarray}
Inserting \eqref{hd0}, \eqref{hd1}, \eqref{hd2} and \eqref{hd3} in %
\eqref{zero}, we get 
\begin{align*}
&\nu \int_{\Omega }\eta ^{2}(1+|Du(x)|^{2}+|Du(x+he_{s})|^{2})^{\frac{p-2}{2}%
}|\tau _{h}Du|^{2}\,dx \\
\leq &\frac{\nu }{2}\int_{\Omega }\eta
^{2}(1+|Du(x)|^{2}+|Du(x+he_{s})|^{2})^{\frac{p-2}{2}}|\tau _{h}Du|^{2}\,dx
\\
&+ c_{\nu ,L_{2}}\Vert 1+|Du|\Vert _{\infty }^{2q-p}\int_{\Omega }|D\eta
|^{2}|\tau _{h}u|^{2}\,dx \\
&+ c_{\nu }|h|^{2}\Vert 1+|Du|\Vert _{\infty }^{2q-p}\int_{\Omega }\eta
^{2}\left( \int_{0}^{1}k(x+\sigma he_{s})d\sigma \right) ^{2}\,dx \\
&+ |h|\Vert 1+|Du|\Vert _{\infty }^{q-1}\left( \int_{\Omega }\eta |D\eta
|\left( \int_{0}^{1}k(x+\sigma he_{s})d\sigma \right)^{\frac{p}{p-1}%
}\,dx\right)^{\frac{p-1}{p}} \\
&\qquad \times \left( \int_{\Omega }\eta |D\eta ||\tau _{h}u|^{p}\,dx\right)
^{\frac{1}{p}}.
\end{align*}%
Reabsorbing the first integral in the right hand side by the left hand side
and recalling the properties of the cut off function $\eta $, we obtain 
\begin{eqnarray*}
&&\int_{B_{\rho }}(1+|Du(x)|^{2}+|Du(x+he_{s})|^{2})^{\frac{p-2}{2}}|\tau
_{h}Du|^{2}\,dx \\
&\leq &\frac{c}{(R-\rho )^{2}}\Vert 1+|Du|\Vert _{\infty
}^{2q-p}\int_{B_{R}}|\tau _{h}u|^{2}\,dx \\
&&+c|h|^{2}\Vert 1+|Du|\Vert _{\infty }^{2q-p}\int_{B_{R}}\left(
\int_{0}^{1}k(x+\sigma he_{s})d\sigma \right) ^{2}\,dx \\
&&+c|h|\Vert 1+|Du|\Vert _{\infty }^{q-1}\left( \int_{B_{R}}\left(
\int_{0}^{1}k(x+\sigma he_{s})d\sigma \right) ^{\frac{p}{p-1}}\,dx\right) ^{%
\frac{p-1}{p}} \\
&&\qquad \times \left( \int_{B_{R}}|\tau _{h}u|^{p}\,dx\right) ^{\frac{1}{p}%
},
\end{eqnarray*}%
where $c=c(\nu ,L_{2},n,N)$. By property (iii) of the finite difference
operator we deduce that 
\begin{eqnarray}
&&\int_{B_\rho} (1+|Du(x)|^2+|Du(x+he_s)|^2)^{\frac{p-2}{2}}|\tau_h Du|^2\,dx
\notag \\
&\le& \frac{c}{(R-\rho)^2}|h|^2 \|1+|Du|\|^{2q-p}_{\infty}\int_{B_{2R}}|D
u|^2\,dx  \notag \\
&&+c|h|^2\|1+|Du|\|^{2q-p}_{\infty}\int_{B_R}\left(\int_0^1 k(x+\sigma
he_s)d\sigma\right)^2\,dx  \notag \\
&&+ c|h|^2\|1+|Du|\|^{q-1}_{\infty}\left(\int_{B_R} \left(\int_0^1
k(x+\sigma he_s)d\sigma\right)^{\frac{p}{p-1}}\,dx\right)^{\frac{p-1}{p}} 
\notag \\
&&\qquad\times \left(\int_{B_{2R}}|D u|^p\,dx\right)^{\frac{1}{p}}.
\label{hd4}
\end{eqnarray}
Dividing \eqref{hd4} by $|h|^{2}$ and using Lemma \ref{Vi} in the left hand
side, we get 
\begin{eqnarray*}
&&\int_{B_{\rho }}\frac{|\tau _{h}V_{p}(Du)|^{2}}{|h|^{2}}\,dx \\
&\leq &\frac{c}{(R-\rho )^{2}}\Vert 1+|Du|\Vert _{\infty
}^{2q-p}\int_{B_{2R}}|Du|^{2}\,dx \\
&&+c\Vert 1+|Du|\Vert _{\infty }^{2q-p}\int_{B_{R}}\left(
\int_{0}^{1}k(x+\sigma he_{s})d\sigma \right) ^{2}\,dx \\
&&+c\Vert 1+|Du|\Vert _{\infty }^{q-1}\left( \int_{B_{R}}\left(
\int_{0}^{1}k(x+\sigma he_{s})d\sigma \right) ^{\frac{p}{p-1}}\,dx\right) ^{%
\frac{p-1}{p}} \\
&&\qquad \times \left( \int_{B_{2R}}|Du|^{p}\,dx\right) ^{\frac{1}{p}}.
\end{eqnarray*}%
Letting $h$ go to $0$, by property (iv) of the finite difference operator,
we conclude 
\begin{eqnarray}
&&\int_{B_\rho} |D V_p(Du)|^2\,dx  \notag \\
&\le& \frac{c}{(R-\rho)^2} \|1+|Du|\|^{2q-p}_{\infty}\int_{B_{2R}}|D
u|^2\,dx+c\|1+|Du|\|^{2q-p}_{\infty}\int_{B_R} k^2(x)\,dx  \notag \\
&&+ c\|1+|Du|\|^{q-1}_{\infty}\left(\int_{B_R} k^{\frac{p}{p-1}%
}(x)\,dx\right)^{\frac{p-1}{p}}\left(\int_{B_{2R}}|D u|^p\,dx\right)^{\frac{1%
}{p}}.  \label{hd6}
\end{eqnarray}
Since $k\in L^{r}$, with $r>n\geq 2\geq \frac{p}{p-1}$ , estimate \eqref{hd6}
implies the conclusion.
\end{proof} 

\section{Proof of Theorem \ref{t:main}}

\label{s:finale} Using the previous results and an approximation procedure,
we can prove of our main result.

\begin{proof}[Proof of Theorem \protect\ref{t:main}]
For $f(x,\xi)$ satisfying the assumptions \eqref{(A2)--(A4)}--\eqref{gap},
let us introduce the sequence 
\begin{equation}  \label{fh}
f_h(x,\xi)=f(x,\xi)+\frac{1}{h}(1+|\xi|^2)^{\frac{ps}{2(s+1)}}.
\end{equation}
Note that $f_h(x,\xi)$ satisfies 
the following set of conditions 
\begin{equation}  \label{growthh}
\frac{1}{h}(1+|\xi|^2)^{\frac{ps}{2(s+1)}} \le f_h(x,\xi)\le
(1+L)\,(1+|\xi|^2)^{\frac{q}{2}},
\end{equation}
\begin{equation}  \label{(A2h)}
\frac{c_1}{h}\,(1+|\xi|^2)^{\frac{ps}{2(s+1)}-2}|\lambda|^2\le \langle
D_{\xi\xi} f_h(x,\xi)\lambda,\lambda\rangle,
\end{equation}
\begin{equation}  \label{(A3h)}
| D_{\xi\xi} f_h(x,\xi)| \le c_2 (1+L)\,(1+|\xi|^2)^{\frac{q-2}{2}},
\end{equation}
\begin{equation}  \label{(A4h)}
|D_{\xi x}f_h(x,\xi)|\le k(x)(1+|\xi|^2)^{\frac{q-1}{2}}
\end{equation}
for some constants $c_1,c_2>0$, for a.e. $x \in \Omega $ and for every $%
\xi\in \mathbb{R}^{N\times n}$.

Now, fix a ball $B_{R}\Subset \Omega$, and let $v_h\in W^{1,\frac{ps}{s+1}}(B_{R}, 
\mathbb{R} ^N)$ be the unique solution to the problem 
\begin{equation}
\min\left\{\int_{B_{R}} f_h(x,Dv)\,dx :\,\, v_h\in u +W^{1,\frac{ps}{s+1}%
}_0(B_{R}, \mathbb{R} ^N)\right\}.
\end{equation}
Since $f_h (x,\xi)$ satisfies \eqref{growthh}, \eqref{(A2h)}, \eqref{(A3h)} %
\eqref{(A4h)} with $k\in L^r$, $r>n$, and \eqref{gap} holds, then by the
result in \cite{EMM1} we have that $v_h \in W^{1,\infty}_{\mathrm{loc}}(B_R)$
and by Theorem \ref{highdiff}, used with $p$ replaced by $p\frac{s}{s+1}$,
we also have $v_h\in W^{2,2}_{\mathrm{loc}}(B_R)$.

Since $f_{h}(x,\xi )$ satisfies \eqref{growthh}, by the minimality of $v_{h}$
we get 
\begin{eqnarray*}
&&\int_{B_{R}}|Dv_{h}|^{\frac{ps}{s+1}}\,dx\leq
c_{s}\int_{B_{R}}a(x)|Dv_{h}|^{p}+c_{s}\int_{B_{R}}\frac{1}{a^{s}(x)}\,dx \\
&\leq &c_{s}\int_{B_{R}}f_{h}(x,Dv_{h})\,dx+c_{s}\int_{B_{R}}\frac{1}{%
a^{s}(x)}\,dx \\
&\leq &c_{s}\int_{B_{R}}f_{h}(x,Du)\,dx+c_{s}\int_{B_{R}}\frac{1}{a^{s}(x)}%
\,dx \\
&=&c_{s}\int_{B_{R}}f(x,Du)\,dx+\frac{c_{s}}{h}\int_{B_{R}}(1+|Du|)^{\frac{ps%
}{s+1}}\,dx+c_{s}\int_{B_{R}}\frac{1}{a^{s}(x)}\,dx \\
&\leq &c_{s}\int_{B_{R}}f(x,Du)\,dx+c_{s}\int_{B_{R}}(1+|Du|)^{\frac{ps}{s+1}%
}\,dx+c_{s}\int_{B_{R}}\frac{1}{a^{s}(x)}\,dx.
\end{eqnarray*}%
Therefore the sequence $v_{h}$ is bounded in $W^{1,\frac{ps}{s+1}}(B_{R})$,
so there exists $v\in u+W_{0}^{1,\frac{ps}{s+1}}(B_{R})$ such that, up to
subsequences, 
\begin{equation}
v_{h}\rightharpoonup v\qquad \text{weakly in}\,\,W^{\frac{ps}{s+1}}(B_{R}).
\label{wconv}
\end{equation}%
On the other hand, we can apply Theorem \ref{t:apriori} to $f_{h}(x,\xi )$
since the assumptions are satisfied, with $b$ replaced by $1+L$. Thus, we
are legitimate to apply estimates \eqref{stimafin} and \eqref{stimafin2} to
the solutions $v_{h}$ to obtain 
\begin{eqnarray}
&&\|Dv_h\|_{L^\infty(B_\rho)}\le C\mathcal{K}_R^{\tilde\vartheta}
\left(\int_{B_R}(1+f_h(x,Dv_h))\,dx\right)^{\tilde\varsigma}  \notag \\
&\le& C\mathcal{K}_R^{\tilde\vartheta}
\left(\int_{B_R}(1+f_h(x,Du))\,dx\right)^{\tilde\varsigma}  \notag \\
&=&C\mathcal{K}_R^{\tilde\vartheta} \left(\int_{B_R}(1+f(x,Du)+\frac{1}{h}%
(1+|Du|^2)^{\frac{ps}{2(s+1)}})\,dx\right)^{\tilde\varsigma}  \notag \\
&\le&C\mathcal{K}_R^{\tilde\vartheta}\left(\int_{B_R}(1+f(x,Du)+(1+|Du|^2)^{%
\frac{ps}{2(s+1)}})\,dx\right)^{\tilde\varsigma},  \label{due}
\end{eqnarray}
with $C,\tilde{\vartheta},\tilde{\varsigma}$ independent of $h$ and $0<\rho
<R$. Therefore, up to subsequences, 
\begin{equation}
v_{h}\rightharpoonup v\qquad \text{weakly* in}\,\,W^{1,\infty }(B_{\rho }).
\label{wconv2}
\end{equation}%
Our next aim is to show that $v=u$. The lower semicontinuity of $u\mapsto
\int_{B_{R}}f(x,Du)$ and the minimality of $v_{h}$ imply 
\begin{eqnarray*}
&&\int_{B_R}f(x,Dv)\,dx\le \liminf_h \int_{B_R}f(x,Dv_h)\,dx\le \liminf_h
\int_{B_R}f_h(x,Dv_h)\,dx \\
&\le& \liminf_h \int_{B_R}f_h(x,Du)\,dx \\
&=&\liminf_h \int_{B_R}(f(x,Du)+\frac{1}{h}(1+|Du|^2)^{\frac{ps}{2(s+1)}%
})\,dx \\
&=& \int_{B_R}f(x,Du)\,dx.
\end{eqnarray*}
The strict convexity of $f$ yields that $u=v$. Therefore passing to the
limit as $h\rightarrow \infty $ in \eqref{due} we get 
\begin{equation*}
\Vert Du\Vert _{L^{\infty }(B_{\rho })}\leq C\mathcal{K}_{R}^{\tilde{%
\vartheta}}\left( \int_{B_{R}}(1+f(x,Du)+(1+|Du|^{2})^{\frac{ps}{2(s+1)}%
})\,dx\right) ^{\tilde{\varsigma}},
\end{equation*}%
i.e. \eqref{fin}. Moreover, we are legitimate to apply estimate %
\eqref{stimafin2} to each $v_{h}$ thus getting 
\begin{eqnarray*}
\int_{B_{\rho }}a(x)(1+|Dv_{h}|^{2})^{\frac{p-2}{2}}|D^{2}v_{h}|^{2}\,dx
&\leq &c\left( \int_{B_{R}}(1+f_{h}(x,Dv_{h}))\,dx\right) ^{\tilde{\varrho}}
\\
&\leq &c\left( \int_{B_{R}}(1+f_{h}(x,Du))\,dx\right) ^{\tilde{\varrho}} \\
&=&c\left( \int_{B_{R}}\left( 1+f(x,Du)+\frac{1}{h}(1+|Du|^{2})^{\frac{ps}{%
2(s+1)}}\right) \,dx\right) ^{\tilde{\varrho}},
\end{eqnarray*}%
where we used the minimality of $v_{h}$ and the definition of $f_{h}(x,\xi )$%
. Since $v_{h}\rightarrow u$ a.e. up to a subsequence, we conclude that 
\begin{eqnarray*}
&&\int_{B_{\rho }}a(x)(1+|Du|^{2})^{\frac{p-2}{2}}|D^{2}u|^{2}\,dx\leq
\liminf_{h}\int_{B_{\rho }}a(x)(1+|Dv_{h}|^{2})^{\frac{p-2}{2}%
}|D^{2}v_{h}|^{2}\,dx \\
&&\quad \leq c\liminf_{h}\left( \int_{B_{R}}\left( 1+f(x,Du)+\frac{1}{h}%
(1+|Du|^{2})^{\frac{ps}{2(s+1)}}\right) \,dx\right) ^{\tilde{\varrho}} \\
&&\quad =\left( \int_{B_{R}}\left( 1+f(x,Du)\right) \,dx\right) ^{\tilde{%
\varrho}},
\end{eqnarray*}%
i.e. \eqref{hdfin}.
\end{proof}


\begin{thebibliography}{99}
\bibitem{bdgp} \textsc{A.K. Balci, L. Diening, R. Giova, A. Passarelli di
Napoli:} {\ Elliptic equations with degenerate weights}, (2020),
arXiv:2003.10380v1.

\bibitem{BCM} \textsc{P. Baroni, M. Colombo, G. Mingione:}{\ Regularity for
general functionals with double phase,}\emph{\ Calc. Var. Partial
Differential}, \textbf{57} (2018),  no. 2, Paper No. 62, 48 pp.

\bibitem{bella-sch} \textsc{P. Bella, M. Sch\"{a}ffner:} {\ Local
boundedness and Harnack inequality for solutions of linear non-uniformly
elliptic equations}, \emph{Comm. Pure App. Math.} (2019) to appear.

\bibitem{belloni-buttazzo} \textsc{M. Belloni, G. Buttazzo:}{\ A survey of
old and recent results about the gap phenomenon in the calculus of
variations,}\emph{\ R. Lucchetti, J. Revalski (Eds.), Recent developements
in well-posed variational problems, Mathematical Applications,}\text{\ }$%
\mathbf{331}${\ (1995), 1-27}.

\bibitem{BiaCupMas} \textsc{\ S. Biagi, G. Cupini, E. Mascolo:} {\
Regularity of quasi-minimizers for non-uniformly elliptic integrals,} \emph{%
J. Math Anal. Appl.}, \textbf{485} (2020), 123838, 20 pp.

\bibitem{CMP1} \textsc{M. Carozza, G. Moscariello, A. Passarelli di Napoli:}{%
\ Higher integrability for minimizers of anisotropic functionals}\emph{\
Discrete Contin. Dyn. Syst. Ser. B}, $\mathbf{11}$ (2009), 43-55.

\bibitem{Cianchi-Mazya2011} \textsc{A. Cianchi, V. Maz'ya:}{\ }Global
Lipschitz regularity for a class of quasilinear elliptic equations, \textit{%
Comm. Partial Differential Equations}, \textbf{36} (2011), 100-133.

\bibitem{colmin} \textsc{M. Colombo, G. Mingione:}{\ Regularity for double
phase variational problems,}\emph{\ Arch. Rat. Mech. Anal.}, \textbf{215}
(2015), 443-496.

\bibitem{colmin2} \textsc{M. Colombo, G. Mingione:}{\ Bounded minimisers of
double phase variational integrals}, \textit{Arch. Rat. Mech. Anal.}, 
\textbf{218} (2015), 219-273.

\bibitem{cruz} \textsc{D. Cruz-Uribe, P. Di Gironimo, C. Sbordone:}{\ On the
continuity of solutions to degenerate elliptic equations}, \emph{J.
Differential Equations} \textbf{250} (2011), 2671-2686.

\bibitem{CGGP} \textsc{G. Cupini, F. Giannetti, R. Giova, A. Passarelli di
Napoli:} {\ Regularity results for vectorial minimizers of a class of
degenerate convex integrals,} \textit{J. Differential Equations}, \textbf{265%
} (2018), 4375-4416.

\bibitem{cupguimas} \textsc{G. Cupini, M. Guidorzi, E. Mascolo:} {\
Regularity of minimizers of vectorial integrals with $p,q$}${-}$growth{,} 
\textit{Nonlinear Anal.}, \textbf{54} (2003), 591-616.

\bibitem{Cupini-Marcellini-Mascolo2012} \textsc{G. Cupini, P. Marcellini, E.
Mascolo:} Local boundedness of solutions to quasilinear elliptic systems, 
\textit{Manuscripta Math.}, \textbf{137}{\ (2012), 287-315.} %

\bibitem{Cupini-Marcellini-Mascolo 2014} \textsc{G. Cupini, P. Marcellini,
E. Mascolo:} {Existence and regularity for elliptic equations under $p,q$}${-%
}${growth,} \textit{Adv. Differential Equations},\textbf{\ 19}{\ (2014),
693-724}.

\bibitem{CupMarMas17} \textsc{G. Cupini, P. Marcellini, E. Mascolo:} {\
Regularity of minimizers under limit growth conditions}, \emph{Nonlinear
Anal.}, \textbf{153} (2017), 294-310.

\bibitem{CupMarMas18} \textsc{G. Cupini, P. Marcellini, E. Mascolo:} {\
Nonuniformly elliptic energy integrals with $p,q$-growth}, \emph{Nonlinear
Anal.}, \textbf{177} (2018), part A, 312-324.

\bibitem{deg} \textsc{E. De Giorgi:} {\ Un esempio di estremali discontinue
per un problema variazionale di tipo ellittico,}\emph{\ Boll. Un. Mat. Ital.,%
} \textbf{1}{\ (1968), 135-137}.

\bibitem{defi-ming} \textsc{C. De Filippis, G. Mingione:} {\ On the
regularity of non-autonomous functionals,} \emph{J. Geom. Anal.} \textbf{30}
(2020), 1584-1626.

\bibitem{defi-ok} \textsc{C. De Filippis, J. Oh:} {\ Regularity for
multi-phase variational problems} \emph{J. Differential Equations}, \textbf{%
267} (2019), 1631-1670.

\bibitem{DiMarco-Marcellini} \textsc{T. Di Marco, P. Marcellini:} A-priori
gradient bound for elliptic systems under either slow or fast growth
conditions, \textit{Calc. Var. Partial Differential Equations}, \textbf{59}
(2020), 26 pp.

\bibitem{Duzaar-Mingione2010} \textsc{F. Duzaar, G. Mingione:} Local
Lipschitz regularity for degenerate elliptic systems, \textit{Ann. Inst. H.
Poincar\'{e} Anal. Non Lin\'{e}aire}, \textbf{27} (2010), 1361-1396.

\bibitem{EMM1} \textsc{M. Eleuteri, P. Marcellini, E. Mascolo.}{\ Lipschitz
estimates for systems with ellipticity conditions at infinity.} \emph{Ann.
Mat. Pura Appl.}, \textbf{195} (2016), 1575-1603.

\bibitem{EMM2} \textsc{M. Eleuteri, P. Marcellini, E. Mascolo:} {\ Lipschitz
continuity for functionals with variable exponents} \emph{\ Rend. Lincei
Mat. Appl.,} \textbf{27} (2016), 61-87.

\bibitem{EMM3} \textsc{M. Eleuteri, P. Marcellini, E. Mascolo:} {\ }%
Regularity for scalar integrals without structure conditions, \textit{%
Advances in Calculus of Variations}, \textbf{13} (2020), 279-300.
https://doi.org/10.1515/acv-2017-0037

\bibitem{esp-leo-pet} \textsc{A. Esposito, F. Leonetti, P.V. Petricca:}{\
Absence of Lavrentiev gap for non-autonomous functionals with $(p,q)-$growth,%
} \emph{\ Adv. Nonlinear Anal.,} \textbf{8} (2019), 73-78.

\bibitem{espleomin} \textsc{L. Esposito, F. Leonetti, G. Mingione:} {\
Regularity results for minimizers of irregular integrals with $(p,q)$ growth}%
, \emph{Forum Mathematicum}, \textbf{14} (2002), 245-272.

\bibitem{espleominp} \textsc{L. Esposito, F. Leonetti, G. Mingione:} {\
Sharp regularity for functionals with $(p,q)$ growth}, \emph{J. Differential
Equations}, \textbf{204} (2004), 5-55.

%


\bibitem{fabes-kenig-serapioni} \textsc{E. Fabes, C. Kenig, R. Serapioni:} {%
\ The local regularity of solutions of degenerate elliptic equations}, \emph{%
Comm. Partial Diff. Eq.}, \textbf{7} (1982), 77-116.

\bibitem{GPdN} \textsc{\ R. Giova, A. Passarelli di Napoli:} {\ Regularity
results for a priori bounded minimizers of non-autonomous functionals with
discontinuous coefficients}, \textit{Adv. Calc. Var.}, \textbf{12} (2019),
85-110.

\bibitem{Giusti} \textsc{E. Giusti:} {\ Direct methods in the calculus of
variations}. World scientific publishing Co. (2003) 50.

\bibitem{Iwaniec} \textsc{T. Iwaniec, L. Migliaccio, G. Moscariello, A.
Passarelli di Napoli}: {\ A priori estimates for nonlinear elliptic complexes%
}, Adv. Differential Equations, \textbf{8} (2003), 513-546.

\bibitem{Iwaniec-Sbordone} \textsc{T. Iwaniec, C. Sbordone:} {\
Quasiharmonic fields,} \emph{Ann. Inst. H. Poincar\'{e} Anal. Non Lin\'{e}%
aire}, \textbf{18} (2001), 519-572.

\bibitem{mar85} \textsc{P. Marcellini:}{\ Approximation of quasiconvex
functions, and lower semicontinuity of multiple integrals,} \emph{%
Manuscripta Math.,} \textbf{51} (1985), no. 1-3, 1-28.

\bibitem{mar89} \textsc{P. Marcellini:}{\ Regularity of minimizers of
integrals in the calculus of variations with non standard growth conditions,}%
\emph{\ Arch. Rational Mech. Anal.}, \textbf{105}{\ (1989) 267-284}.

\bibitem{mar91} \textsc{P. Marcellini:}{\ Regularity and existence of
solutions of elliptic equations with $p,q$-growth conditions,} \emph{\ J.
Differential Equations}, \textbf{90} {(1991), 1-30}.

\bibitem{mar96} \textsc{P. Marcellini:}{\ Everywhere regularity for a class
of elliptic systems without growth conditions ,}\emph{\ Ann. Scuola Norm.
Sup. Pisa Cl. Sci.}, \textbf{23}{\ (1996), 1-25}.

\bibitem{MarcelliniDiscrContDinSystems2019} \textsc{P. Marcellini:}{\ }%
Regularity under general and $p,q-$growth conditions, \textit{Discrete Cont.
Dinamical Systems Series S}, \textbf{13} (2020), 2009-2031.

\bibitem{MarcelliniNonAnal2019} \textsc{P. Marcellini:}{\ }A variational
approach to parabolic equations under general and $p,q-$growth conditions, 
\textit{Nonlinear Anal.}, \textbf{194} (2020), 111456, 17 pp.

\bibitem{MarcelliniJMAA2020} \textsc{P. Marcellini:} Growth conditions and
regularity for weak solutions to nonlinear elliptic pdes, \textit{J. Math.
Anal. Appl.}, 2020, to appear. \ https://doi.org/10.1016/j.jmaa.2020.124408

\bibitem{mar-papi} \textsc{P. Marcellini, G. Papi:} {\ Nonlinear elliptic
systems with general growth,}\emph{\ J. Differential Equations},{\ }\textbf{%
221}{\ (2006), 412-443. }

\bibitem{Mingione} \textsc{G. Mingione:}{\ Regularity of minima: an
invitation to the dark side of the calculus of variations,}\emph{\ Appl.
Math.,}\textbf{\ 51}{\ (2006), 355-426}.

\bibitem{moo-sa} \textsc{C. Mooney, O. Savin :}{\ Some singular minimizers
in low dimensions in the calculus of variations,}\emph{\ Arch. Ration. Mech.
Anal.},\textbf{\ 22} {\ (2016), 1-22} %

\bibitem{moser} \textsc{J. Moser:}{\ A new proof of De Giorgi's theorem
concerning the regularity problem for elliptic differential equations,} 
\emph{Comm. Pure Appl. Math.,} \textbf{13} (1960), 457-468.

\bibitem{PdN14-1} \textsc{A. Passarelli di Napoli:} {\ Higher
differentiability of minimizers of variational integrals with Sobolev
coefficients}, \emph{Adv. Calc. Var.}, \textbf{7} (2014), 59-89.

\bibitem{PdN14-2} \textsc{A. Passarelli di Napoli:} {\ Higher
differentiability of solutions of elliptic systems with Sobolev
coefficients: the case $p=n=2$}, \emph{Pot. Anal.}, \textbf{41} (2014),
715-735.

\bibitem{PdN15} \textsc{A. Passarelli di Napoli:} {\ Regularity results for
non-autonomous variational integrals with discontinuous coefficients}, \emph{%
Atti Accad. Naz. Lincei, Rend. Lincei Mat. Appl.},{\ }\textbf{26} (2015),
(4), 475-496.

\bibitem{pingen} \textsc{M. Pingen:}{\ Regularity results for degenerate
elliptic systems,}\emph{\ }\textit{Ann. Inst. H. Poincar\'{e} Anal. Non Lin%
\'{e}aire}, \textbf{25} (2008), 369-380.

\bibitem{sverak} \textsc{V. $\check{\text{S}}$ver\'{a}k, X.Yan:} {\ A
singular minimizer of a smooth strongly convex functional in three
dimensions,}\emph{\ Calc. Var. Partial Differential Equations},{\ \textbf{10}%
\ (2000) 213-221}

\bibitem{Trud} \textsc{N.S. Trudinger}: {\ On the regularity of generalized
solutions of linear, non-uniformly elliptic equations}, \emph{Arch. Rational
Mech. Anal.}, \textbf{42} (1971), 42-50.

\bibitem{Trud2} \textsc{N.S. Trudinger:}{\ Linear elliptic operators with
measurable coefficients}, \emph{Ann. Scuola Norm. Sup. Pisa Cl. Sci.}, 
\textbf{27} (1973), 265-308.

\bibitem{zhikov} \textsc{V.V. Zhikov:}{\ On Lavrentiev phenomenon,} \emph{\
Russian J. Math. Phys.}, \textbf{3} (1995), 249-269.
\end{thebibliography}
\end{document}